\documentstyle[amssymb,amsfonts,12pt]{amsart}

\newtheorem{theorem}{Theorem}[section]
\newtheorem{lemma}[theorem]{Lemma}
\newtheorem{corollary}[theorem]{Corollary}
\newtheorem{proposition}[theorem]{Proposition}
\newtheorem{definition}[theorem]{Definition}

\theoremstyle{remark}
\newtheorem{remark}[theorem]{Remark}
\newtheorem*{ack*}{Acknowledgment}

\def\D{{\operatorname{Dec}}}\def\Di{{\operatorname{Part}}}

\def\F{{\mathcal F}}\def\B{{\mathcal B}}
\def\R{{\mathbb R}}

\def\N{{\mathbb N}}
\def\C{{\mathbb C}}

\def\Z{{\mathbb Z}}
\def\P{{\mathcal P}}\def\nint{\mathop{\diagup\kern-13.0pt\int}}
\def\H{{\mathcal W}}
\def\W{{\mathbb P}}

\def\dist{{\operatorname{dist}}}

\def\bas{\begin{align*}}
\def\eas{\end{align*}}
\def\bi{\begin{itemize}}
\def\ei{\end{itemize}}
\newenvironment{proof}{\noindent {\bf Proof} }{\endprf\par}
\def \endprf{\hfill  {\vrule height6pt width6pt depth0pt}\medskip}
\def\emph#1{{\it #1}}

\begin{document}
\author{Jean Bourgain}
\address{School of Mathematics, Institute for Advanced Study, Princeton NJ}
\email{bourgain@@math.ias.edu}
\author{Ciprian Demeter}
\address{Department of Mathematics, Indiana University,  Bloomington IN}
\email{demeterc@@indiana.edu}
\thanks{The first author is partially supported by the NSF grant DMS-1301619. The second  author is partially supported  by the NSF Grant DMS-1161752}

\title[The $l^2$ Decoupling Theorem]{A study guide for the $l^2$ Decoupling Theorem}

\begin{abstract}
This paper contains a  detailed, self contained  and  more streamlined proof of our $l^2$ decoupling theorem for hypersurfaces from \cite{BD3}. We hope this  will serve as a good warm up for the readers interested in understanding the proof of Vinogradov's mean value theorem from \cite{BDG}.
\end{abstract}
\maketitle

\section{The theorem}

Consider the truncated (elliptic) paraboloid in $\R^n$
$$\W^{n-1}:=\{(\xi_1,\ldots,\xi_{n-1},\xi_1^2+\ldots+\xi_{n-1}^2):\;0\le \xi_i\le 1\}.$$
For each  cube $Q$ in $[0,1]^{n-1}$ and $g:Q\to\C$ define the extension operator $E^{(n)}_Q=E_Q$ as follows
$$E_Qg(x)=\int_Qg(\xi_1,\ldots,\xi_{n-1})e(\xi_1x_1+\ldots+ \xi_{n-1}x_{n-1}+(\xi_1^2+\ldots+\xi_{n-1}^2)x_n)d\xi,$$
where $e(z)=e^{2\pi iz}$,
$$\xi=(\xi_1,\ldots,\xi_{n-1}),$$
and
$$x=(x_1,\ldots,x_n).$$
This can be interpreted as the Fourier transform $\widehat{gd\sigma}$, where the measure $d\sigma$ is the lift of the Lebesgue measure from $[0,1]^{n-1}$ to the paraboloid.
When $Q=[0,1]^{n-1}$, we will sometimes write $Eg$ for $E_{[0,1]^{n-1}}g$.

\medskip

We will use the letters $Q,q$ to denote cubes  on the frequency side $[0,1]^{n-1}$. We will use the letters $B,\Delta$ to denote cubes on the spatial side $\R^n$. Throughout the whole paper we can and will implicitly assume that all cubes have side length in $2^{\Z}$. This in particular will place (harmless) restrictions on various  parameters such as $\delta,\sigma,R$, that we will not bother to write down explicitly. Thanks to this assumption we will be able  to partition (rather than use finitely overlapping covers) large cubes into smaller cubes. Given a  cube $Q\subset [0,1]^{n-1}$ with side length $l(Q)\in 2^{-\N}$ and  $\alpha\in 2^{-\N}$ smaller than $l(Q)$,  we will denote by $\Di_{\alpha}(Q)$  the (unique) partition of $Q$ using cubes $Q_{\alpha}$ of side length $\alpha$. A similar notation will occasionally be used for spatial cubes $B$.
\medskip

We will write  $B=B(c_B,R)$ for the cube in $\R^n$  centered at $c_B$ and with side length $l(B_R)=R$ and we will introduce the associated weight
$$w_{B}(x)= \frac{1}{(1+\frac{|x-c_B|}{R})^{100n}}.$$
The exponent $100n$ is chosen large enough to guarantee various integrability requirements.  We will see that Theorem \ref{main:steamline} remains  true for any larger exponent $E\ge 100n$, and the implicit bounds  will depend on $E$. This observation will allow us to run our induction argument, as explained in Section \ref{brief}.

For a positive weight $v:\R^n\to[0,\infty)$ and for $f:\R^n\to\C$ we define the weighted integral
$$\|f\|_{L^p(v)}=(\int_{\R^n}|f(x)|^pv(x)dx)^{1/p}.$$
\bigskip

For $2\le p\le \infty$ and $\delta\in 4^{-\N}$, let $\D(\delta,p)=\D_n(\delta,p)$ be the smallest constant such that the inequality
$$\|Eg\|_{L^p(w_B)}\le \D(\delta,p)(\sum_{Q\in\Di_{\delta^{1/2}}([0,1]^{n-1})}\|E_Qg\|_{L^p(w_B)}^2)^{1/2},$$
holds for every cube $B\subset \R^n$ with side length  $\delta^{-1}$ and  every $g:[0,1]^{n-1}\to\C$.
\medskip

The $l^2$ decoupling theorem proved in \cite{BD3} reads as follows. We refer the reader to \cite{BD3} for a few applications that motivate the theorem.

\begin{theorem}
\label{main:steamline}
We have the following sharp (up to $\delta^{-\epsilon}$ losses) upper bound for $\D_n(\delta,p)$
$$\D_n(\delta,p)\lesssim_{\epsilon,p,n}\delta^{-\epsilon}$$ if $2\le p\le \frac{2(n+1)}{n-1}.$
The implicit constant depends on $\epsilon,p,n$ but not on $\delta$.

\end{theorem}

We will present a rather detailed argument for this theorem. Essentially, we rewrite our original argument from  \cite{BD3} using a more streamlined approach. This  approach has started to take shape in our subsequent papers on decouplings and has gotten to this final form in the joint work with Guth \cite{BDG}. One new feature of our argument compared to  \cite{BD3} is that we avoid the special interpolation  from \cite{BD3}, that relied on wave-packet decomposition. Another one is that we use the multilinear Kakeya inequality, rather than the multilinear restriction theorem. The argument we describe here also clarifies various technical aspects of the theory, such as the role of the weights $w_B$ and the (essentially) locally constant behavior of Fourier transforms of measures supported on caps on the paraboloid.
\medskip

We hope the argument will be  accessible to experts outside the area of harmonic analysis. We believe this  will serve as a warm up for the readers interested in understanding the proof of Vinogradov's mean value theorem from \cite{BDG}.
\medskip

A brief summary of the argument is presented in Section \ref{brief}. The most important sections are the last two. The details from the remaining sections may be skipped at the first reading.

\begin{ack*}
The authors are grateful to Zane Li and Terry Tao for pointing out a few inaccuracies in an earlier version of this manuscript.
\end{ack*}

\bigskip

\section{More notation}
\label{not}

Throughout the paper we will write $A\lesssim_{\upsilon}B$ to denote the fact that $A\le CB$ for a certain implicit constant $C$ that depends on the parameter $\upsilon$. Typically, this parameter is either $\epsilon,\nu$ or $K$. The implicit constant will never depend on the scale $\delta$, on the spatial cubes we integrate over, or on the function $g$. It will however most of the times depend on the degree $n$ and on the Lebesgue index $p$. Since these can  be thought of as being fixed parameters, we will in general not write $\lesssim_{p,n}$.

We will denote by $B_R$ an arbitrary cube in $\R^n$ of side length  $l(B_R)=R$. We use the following two notations for averaged integrals
$$\nint_B F=\frac1{|B|}\int_BF,$$
$$\|F\|_{L^p_\sharp(w_B)}=(\frac1{|B|}\int|F|^pw_B)^{1/p}.$$Given a function $\eta$ on $\R^n$ and a cube $B=B(c,R)$ in $\R^n$, we will use the rescaled version
$$\eta_B(x)=\eta(\frac{x-c}{R}).$$
$|A|$ will refer to either the cardinality of $A$ if $A$ is finite, or to its Lebesgue measure if $A$ has positive measure.

We will sometimes write $\langle f,g\rangle$ for the inner product $\int f\bar{g}$.

\bigskip

\section{A brief description of the argument}
\label{brief}

We use  two types of mechanisms to decouple. One is  the $L^2$ decoupling (Section \ref{L2}). This is very basic, it relies just on Hilbert space orthogonality, but it is  nevertheless very efficient. It decouples into frequency cubes whose side length  is as small as permitted by the uncertainty principle, namely equal to the reciprocal of the side length of the spatial cube. The second mechanism is a multilinear decoupling that relies on the multilinear Kakeya inequality, see  Theorem \ref{nt4}. Combining these with multiple iterations leads to the multiscale inequality \eqref{jhfvheruyg7erigt7t86fo}. This inequality has a very simple form when $2\le p\le \frac{2n}{n-1}$, and a short warm up argument is presented in the end of Section \ref{sec:it} to prove Theorem \ref{main:steamline} in this range.

For the general case, the argument will go as follows. We will introduce a family of constants $\D_n(\delta,p,\nu,m)$ and will show in Section \ref{s10} that they dominate  $\D_n(\delta,p)$. On the other hand, in the last section we use \eqref{jhfvheruyg7erigt7t86fo} to show that each $\D_n(\delta,p,\nu,m)$ can be controlled by a combination of powers of $\delta$ and some power of $\D_n(\delta,p)$, see \eqref{dschjgdshdcyfyuvtioduewuyt}. This inequality represents an improvement over the trivial estimate
$\D_n(\delta,p,\nu,m)\lesssim \D_n(\delta,p).$
By playing the two bounds (\eqref{dschjgdshdcyfyuvtioduewuyt} and \eqref{dbhcgvbdduiqwyfyueeeoiowuydrtuiwqo}) against each other we arrive at the desired upper bound
$$\D_n(\delta,p)\lesssim_\epsilon\delta^{-\epsilon}.$$

An unfortunate technicality is the fact that  we will need to work with the family of weights for a cube $B=B(c,R)$ in $\R^n$
$$w_{B,E}(x)=\frac1{(1+\frac{|x-c|}{R})^{E}}.$$
Here $E\ge 100n$. For each such exponent $E$ we will let as before $\D_n(\delta,p,E)$ denote the smallest constant that guarantees the following inequality for each $g,B=B_{\delta^{-1}}$
$$\|Eg\|_{L^p(w_{B,E})}\le \D_n(\delta,p,E)(\sum_{Q\in\Di_{\delta^{1/2}}([0,1]^{n-1})}\|E_Qg\|_{L^p(w_{B,E})}^2)^{1/2}.$$
All the quantities that will depend on weights will implicitly depend on $E$. This includes $\D_n(\delta,p,\nu,m)$, $D_t(q,B^r,g)$ and
$A_p(q,B^r,s,g)$. Sometimes we will suppress the dependence on $E$ and will understand implicitly that the inequality is true for all $E\ge 100n$. The weight $w_{B,E}$ will always be the same on both sides of a given inequality. The implicit constants will depend on $E$ but that is completely harmless.

We will prove Theorem \ref{main:steamline} using induction on the dimension $n$. We set a superficially stronger induction hypothesis, namely we will assume that
$$\D_{n-1}(\delta,p,E)\lesssim_{\epsilon,E}\delta^{-\epsilon}$$
for each $2\le p\le \frac{2n}{n-2}$ and each $E\ge 100(n-1)$. We will use this to prove that
$$\D_{n}(\delta,p,E)\lesssim_{\epsilon,E}\delta^{-\epsilon}$$
for each $2\le p\le \frac{2(n+1)}{n-1}$ and each $E\ge 100n$.
The reason for such a hypothesis is coming from inequality
\eqref {huregyerwiout8rtut9u56uy956iy89790}, which essentially uses the lower dimensional constant $\D_{n-1}(\delta,p,F)$ to make a statement about $\D_{n}(\delta,p,E)$. Larger dimensions demand higher values of $E$ due to integrability requirements.
\bigskip

\section{A few useful inequalities}

One technical challenge  in the proof of Theorem \ref{main:steamline} is to preserve the exponent $E$ for the weights  $w_B$ involved in various inequalities.

A key, easy to check property of the weights $w_B=w_{B,E}$ that will be used extensively is the following inequality
\begin{equation}
\label{jgiou tiguyt reym89r--9048-0023=-1r90-8=4-293=-0}
1_B\lesssim \sum_{\Delta\in \B}w_{\Delta}\lesssim w_B,
\end{equation}
valid for all cubes $B$ with  $l(B)=R$ and all finitely overlapping covers $\B$ of $B$ with cubes $\Delta$ of (fixed) side length $1\le R'\le R$. The implicit constants in \eqref{jgiou tiguyt reym89r--9048-0023=-1r90-8=4-293=-0} will (harmlessly) depend on $E$, but crucially, they will be independent of $R,R'$.
\medskip

We will find extremely useful the following simple result.

\begin{lemma}
\label{nl9}
Let $\H$ be the collection of all weights, that is, positive, integrable functions on $\R^n$. Fix $R>0$.
Fix $E$. Let $O_1,O_2:\H\to[0,\infty]$ have the following four properties:

(W1) $O_1(1_B)\lesssim O_2(w_{B,E})$ for all cubes $B\subset \R^n$ of side length $R$

(W2) $O_1(\alpha u+\beta v)\le \alpha O_1(u)+\beta O_1(v)$, for each $u,v\in\H$ and $\alpha,\beta>0$

(W3) $O_2(\alpha u+\beta v)\ge \alpha O_2(u)+\beta O_2(v)$, for each $u,v\in\H$ and $\alpha,\beta>0$

(W4) If $u\le v$ then $O_i(u)\le O_i(v)$.

Then $$O_1(w_{B,E})\lesssim O_2(w_{B,E})$$
for each cube $B$ with side length $R$. The implicit constant is independent of $R,B$ and only depends on the implicit constant from (W1), on $E$ and $n$.
\end{lemma}

We will sometimes be able to check a stronger assumption than (W1), where $O_2(w_B)$ is replaced with $O_2(\eta_B)$ for some rapidly decreasing function $\eta$.
\medskip

\begin{proof}
Let $\B$ be a finitely overlapping cover of $\R^n$ with cubes $B'=B'(c_{B'},R)$.
It suffices to note that
$$w_B(x)\lesssim \sum_{B'\in\B}1_{B'}(x)w_B(c_{B'})$$
and that
$$\sum_{B'\in\B}w_{B'}(x)w_B(c_{B'})\lesssim w_B(x).$$
\end{proof}

\begin{remark}
\label{aha}
It is rather immediate that for each $f$
$$O_1(v):=\|f\|_{L^p(v)}^p$$
satisfies (W2) and (W4). Also, for fixed $p\ge 2$ and $f_i$, Minkowski's inequality in $l_{\frac{2}p}$ shows that
$$O_2(v):=(\sum_{i}\|f_i\|_{L^p(v)}^2)^{\frac{p}{2}}$$
satisfies (W3) and (W4). Most applications of Lemma \ref{nl9} will use this type of operators.
\end{remark}

We close this section with the following reverse H\"older inequality.
\begin{corollary}
For each $q\ge p\ge 1$, each cube $Q\subset [0,1]^{n-1}$ with $l(Q)=\frac1{R}$ and each cube $B$ in $\R^n$ with  $l(B)=R$ we have
\begin{equation}
\label{ nghbugtrt90g0-er9t-9}
\|E_{Q}g\|_{L^q_{\sharp}(w_{B,E})}\lesssim \|E_{Q}g\|_{L^p_\sharp(w_{B,\frac{Ep}{q}})},
\end{equation}
with the implicit constant independent of $R$, $Q$, $B$ and $g$.
\end{corollary}
\begin{proof}
Let $\eta$ be a positive smooth function on $\R^n$  satisfying $1_{B(0,1)}\le \eta_{B(0,1)}$ and such that the Fourier transform of $\eta^{\frac1p}$ is  supported on the cube $B(0,1)$.
We can thus write
$$\|E_{Q}g\|_{L^q(B)}\le \|E_{Q}g\|_{L^q(\eta_B^{\frac{q}{p}})}=\|\eta_B^{\frac1p}E_{Q}g\|_{L^q(\R^n)}.$$
Let $\theta$ be a Schwartz function which  equals to 1 on the cube $B(0,10)$.
Since the Fourier transform of $\eta_B^{\frac1p}E_{Q}g$ is supported in the cube $3Q$ (having the same center as $Q$ and side length three times as large), we have that
$$\eta_B^{\frac1p}E_{Q}g=(\eta_B^{\frac1p}E_{Q}g)*\widehat{\theta_Q}$$
and thus, by Young's inequality we can write
$$\|\eta_B^{\frac1p}E_{Q}g\|_{L^q(\R^n)}\le \|\eta_B^{\frac1p}E_{Q}g\|_{L^p(\R^n)}\|\widehat{\theta_Q}\|_{L^r(\R^n)}\lesssim R^{-n/r'}\|E_{Q}g\|_{L^p(\eta_B)}.$$
Here
$$\frac{1}q=\frac1{p}+\frac1{r}-1=\frac1p-\frac1{r'}.$$
Now, following the notation and ideas from the proof of Lemma \ref{nl9} we may use the previous inequalities to write
$$\int |E_Qg|^qw_{B,E}\lesssim \sum_{B'\in\B}w_{B,E}(c_{B'})\int_{B'} |E_Qg|^q$$
$$\lesssim R^{-\frac{nq}{r'}}\sum_{B'\in\B}w_{B,E}(c_{B'})\|E_{Q}g\|_{L^p(\eta_{B'})}^q$$
$$\lesssim R^{-\frac{nq}{r'}}(\sum_{B'\in\B}[w_{B,E}(c_{B'})]^{\frac{p}{q}}\|E_{Q}g\|_{L^p(\eta_{B'})}^p)^{\frac{q}{p}}$$
$$\lesssim R^{-\frac{nq}{r'}}(\int |E_Qg|^pw_{B,\frac{Ep}{q}})^{\frac{q}p}.$$
\end{proof}

\begin{remark}
\label{aha1}
Note that there is a loss in regularity in \eqref{ nghbugtrt90g0-er9t-9}, as the weight exponent is $\frac{Ep}{q}$ on the right hand side. A simple example shows that this exponent is optimal.  This will later cause some minor technicalities. In particular, it will force us to use the smaller weight $w_{\Delta,10E}$ (as opposed to $w_{\Delta,E}$)
in \eqref{ nbo wru45y9d0=1-=c0214=-`=2e4=23-4}. This will in turn allow us to go from \eqref{ fu9ety8290-78-=1=23 -r02v9=--e32} to \eqref{ fu9ety8290-78-=1=23 -r02v9=--e321} by using \eqref{ nghbugtrt90g0-er9t-9} for indices $p$ and $2$.
\end{remark}

\section{An equivalent formulation}

For $\delta<1$ and $Q\subset [0,1]^{n-1}$ define the $\delta$-neighborhood of $\W^{n-1}$ above $Q$ to be
$$N_{\delta}(Q)=\{(\xi_1,\ldots,\xi_{n-1},\xi_1^2+\ldots+\xi_{n-1}^2+t):\;\xi_i\in Q\text{ and }0\le t\le \delta\}.$$
For each $f:\R^n\to\C$ and $R\subset \R^n$ denote by $f_{R}$ the Fourier restriction of $f$ to $R$
$$f_R(x)=\int_R\widehat{f}(\xi)e(x\cdot\xi)d\xi.$$
In this section we will make repeated use of the following inequalities, where $B_R$ will refer to the cube centered at the origin in $\R^n$
\begin{equation}
\label{ryyy878t0934=0=4-56k90u=-56i90-870i=81}
w_{B_R,E}*(\frac1{(R')^n}w_{B_{R'},E})\lesssim w_{B_R,E},\;\;R'\le R
\end{equation}
and, when $n=2$
\begin{equation}
\label{ryyy878t0934=0=4-56k90u=-56i90-870i=82}
w_{B_R,E}(x_1,x_2)\le \left (\frac{1}{1+\frac{|x_1|}{R}}\right)^{E_1}\left (\frac{1}{1+\frac{|x_2|}{R}}\right)^{E_2}, \;\;E_1+E_2\le E.
\end{equation}
\bigskip

We will need the following alternate form of decoupling when we will derive inequality \eqref{huregyerwiout8rtut9u56uy956iy89790}.
\begin{theorem}
\label{alternative}
For each $E\ge 100n$,   the following statement is true for each $F\ge \Gamma_n(E)$, where $\Gamma_n(E)$ is a large enough constant depending on $E$ and $n$.
For $p\ge 2$, each $f:\R^n\to\C$ with Fourier transform supported in $N_{1/R}([0,1]^{n-1})$ and for each cube $B_R\subset \R^n$ we have
$$\|f\|_{L^p(w_{B_R,E})}\lesssim \D_n(R^{-1},p,F)(\sum_{Q\in\Di_{R^{-1/2}}([0,1]^{n-1})}\|f_{N_{1/R}(Q)}\|_{L^p(w_{B_R,E})}^2)^{1/2}.$$
\end{theorem}
\begin{proof}
To simplify notation we will show the computations when $n=2$. In this case $\Gamma_2(E)=2E+2$ will suffice.

Using Remark \ref{aha} it will suffice to prove
$$\|f\|_{L^p({B_R})}\lesssim \D_2(R^{-1},p,F)(\sum_{Q\in\Di_{R^{-1/2}}([0,1])}\|f_{N_{1/R}(Q)}\|_{L^p(w_{B_R,E})}^2)^{1/2}.$$
Due to translation/modulation invariance we may assume $B_R$ to be centered at the origin.

A change of variables allows us to write
$$f(x_1,x_2)=\int_{N_{1/R}([0,1])} \widehat{f}(\xi)e(\xi\cdot x)d\xi=$$$$\sum_{Q\in \Di_{R^{-1/2}}([0,1])}\int_{Q\times[0,\frac1R]}\widehat{f}(s,s^2+t)e(sx_1+s^2x_2)e(tx_2)dsdt.$$
Next, combining this with the Taylor expansion
$$e(tx_2)=\sum_{j\ge 0}\frac{(2\pi)^j}{j!}(\frac{2ix_2}{R})^j(\frac{Rt}{2})^j$$
we can write for $x\in B_R$
\begin{equation}
\label{uyugnugy84587ty458t0-899y085605690}
|f(x)|\le\sum_{j\ge 0}\frac{(4\pi)^j}{j!}|\sum_{Q\in \Di_{R^{-1/2}}([0,1])}E_Qg_j(x)|,
\end{equation}
where
$$g_j(s)=\int_{0}^{R^{-1}}\widehat{f}(s,s^2+t)(\frac{Rt}{2})^jdt.$$
Obviously \eqref{uyugnugy84587ty458t0-899y085605690} leads to the following inequality
$$\|f\|_{L^p({B_R})}\le \D_2(R^{-1},p,F)\sum_{j\ge 0}\frac{(4\pi)^j}{j!}(\sum_{Q\in \Di_{R^{-1/2}}([0,1])}\|E_Qg_j\|_{L^p(w_{B_R,F})}^2)^{1/2}.$$
It remains to prove that (note that we have $F$ on the left and $E$ on the right)
$$
\|E_Qg_j\|_{L^p(w_{B_R,F})}\lesssim \|f_{N_{1/R}(Q)}\|_{L^p(w_{B_R,E})},$$
uniformly over $j\ge 0$.

An easy computation allows us to assume $Q=[0,R^{-1/2}]$. Indeed, translating $[u,u+R^{-1/2}]$ to $[0,R^{-1/2}]$ on the frequency side will  replace $(x_1,x_2)$ with $(x_1+2ux_2,x_2)$ on the spatial side. Note that when $0\le u\le 1$ these shear transformations affect the weights $w_{B}$ only negligibly.

We start by writing
$$\|E_Qg_j\|_{L^p(w_{B_R,F})}^p\sim \int\|E_Qg_j\|_{L^p_\sharp({B(y,R)})}^pw_{B_R,F}(y)dy.$$
Recall that
$$E_Qg_j(x)=\int_{N_{1/R}(Q)}\widehat{f}(\xi)(\frac{R(\xi_2-\xi_1^2)}2)^je((\xi_1^2-\xi_2)x_2)e(\xi\cdot x)d\xi$$
For $x\in B(y,R)$ we write
$$e((\xi_1^2-\xi_2)x_2)=e((\xi_1^2-\xi_2)y_2)e((\xi_1^2-\xi_2)(x_2-y_2))$$
and apply another Taylor expansion for $e((\xi_1^2-\xi_2)(x_2-y_2))$ to arrive at
$$|E_Qg_j(x)|\le \sum_{k\ge 0}\frac{(4\pi)^k}{k!}\left|\int_{N_{1/R}(Q)}\widehat{f}(\xi)(\frac{R(\xi_2-\xi_1^2)}2)^{j+k}e((\xi_1^2-\xi_2)y_2)e(\xi\cdot x)d\xi\right|.$$
It now remains to prove
$$\int \left\|\int_{N_{1/R}(Q)}\widehat{f}(\xi)(\frac{R(\xi_2-\xi_1^2)}2)^{j}e((\xi_1^2-\xi_2)y_2)e(\xi\cdot x)d\xi\right\|_{L^p_\sharp({B(y,R)})}^pw_{B_R,F}(y)dy\lesssim$$$$ \|f_{N_{1/R}(Q)}\|_{L^p(w_{B_R,E})}^p,$$
uniformly over $j\ge 0$.
\medskip

We write
\smallskip

$$\int_{N_{1/R}(Q)}\widehat{f}(\xi)(\frac{R(\xi_2-\xi_1^2)}2)^{j}e((\xi_1^2-\xi_2)y_2)e(\xi\cdot x)d\xi=\int \widehat{F}(\xi)m_j(\xi)e(\xi_1x_1+\xi_2(x_2-y_2))d\xi.$$
where
$$m_j(\xi)=m_{j,y_2}(\xi)=e(\xi_1^2y_2)(\frac{R(\xi_2-\xi_1^2)}{2})^{j}1_{[0,1/2]}(\frac{R(\xi_2-\xi_1^2)}2)\eta(R^{1/2}\xi_1)\eta(R\xi_2),$$
$\eta$ is a Schwartz function equal to 1 on $[-2,2]$ and supported in $[-3,3]$, and
$$F=f_{N_{1/R}(Q)}.$$

 Let $M_j(t)$ be a compactly supported Schwartz function which agrees with $t^j$ on $[0,1/2]$ and satisfies the derivative bound
\begin{equation}
\label{tu54t458568796788-i9078=-7809-7809-=7=354=656=45}
\|\frac{d^k}{dt^k}M_j\|_{L^\infty(\R)}\lesssim_k1,
\end{equation}
uniformly over $j\ge 0$, for each $k\ge 0$.

Note  that we can also write
\smallskip

$$
\int_{N_{1/R}(Q)}\widehat{f}(\xi)(\frac{R(\xi_2-\xi_1^2)}2)^{j}e((\xi_1^2-\xi_2)y_2)e(\xi\cdot x)d\xi=\int \widehat{F}(\xi)\tilde{m}_j(\xi)e(\xi_1x_1+\xi_2(x_2-y_2))d\xi
$$
where
$$\tilde{m}_{j,y_2}(\xi)=\tilde{m}_j(\xi)=e(\xi_1^2y_2)M_j(\frac{R(\xi_2-\xi_1^2)}2)\eta(R^{1/2}\xi_1)\eta(R\xi_2).$$
Applying H\"older we arrive at
$$\int \left\|\int   \widehat{F}(\xi)\tilde{m}_j(\xi)e(\xi_1x_1+\xi_2(x_2-y_2))d\xi        \right\|_{L^p_\sharp({B(y,R)})}^pw_{B_R,F}(y)dy\lesssim $$
$$\int \int |F|^p*|\widehat{\tilde{m}_j}|(x)R^{-2}1_{B_R}(x_1-y_1,x_2)w_{B_R,F}(y)dxdy=$$
$$\int|F|^p(x')\left[\int\int|\widehat{\tilde{m}_j}|(x-x')R^{-2}1_{B_R}(x_1-y_1,x_2)w_{B_R,F}(y)dxdy\right]dx'.$$
It remains to show that
$$\int|\widehat{\tilde{m}_j}|*(R^{-2}1_{B_R})(y_1-x_1,-x_2)w_{B_R,F}(y)dy\lesssim w_{B_R,E}(x).$$
In fact, we will prove a slightly stronger inequality
\begin{equation}
\label{ryyy878t0934=0=4-56k90u=-56i90-870i=8}
\int|\widehat{\tilde{m}_j}|*(R^{-2}1_{B_R})(y_1-x_1,-x_2)w_{B_R,F}(y)dy\lesssim (1+\frac{|x_1|}{R})^{-E}(1+\frac{|x_2|}{R})^{-E}.
\end{equation}

An easy computation using \eqref{tu54t458568796788-i9078=-7809-7809-=7=354=656=45} shows that for each $s_1,s_2\ge 0$
$$\|\partial_{\xi_1}^{s_1}\partial _{\xi_2}^{s_2}\tilde{m}_j\|_{L^\infty}\lesssim_{s_1,s_2} (R^{1/2}+\frac{|y_2|}{R^{1/2}})^{s_1}R^{s_2}.$$
Combining this with the fact that $\tilde{m}_j$ is compactly supported in $[-R^{-1/2},R^{-1/2}]\times [-R^{-1}\times R^{-1}]$ leads, via repeated integration by parts, to the following estimate for the Fourier transform
$$|\widehat{\tilde{m}_j}(x_1,x_2)|\le \phi_1(x_1)\phi_2(x_2)$$
where
\begin{equation}
\label{ryyy878t0934=0=4-56k90u=-56i90-870i=83}
\phi_1(x_1)\lesssim_{s_1}\frac1{R^{1/2}}\left(\frac{1}{1+\frac{|x_1|}{R^{1/2}+R^{-1/2}|y_2|}}\right)^{s_1}
\end{equation}
and
\begin{equation}
\label{ryyy878t0934=0=4-56k90u=-56i90-870i=84}
\phi_2(x_2)\lesssim_{s_2}\frac1{R}\left(\frac1{1+\frac{|x_2|}{R}}\right)^{s_2}.
\end{equation}
Let $I_R=[-R/2,R/2]$ and recall that $B_R=I_R\times I_R$. Using \eqref{ryyy878t0934=0=4-56k90u=-56i90-870i=84} and \eqref{ryyy878t0934=0=4-56k90u=-56i90-870i=81} ($n=1$) we may now write
$$\int(|\widehat{\tilde{m}_j}|*(R^{-2}1_{B_R}))(y_1-x_1,-x_2)w_{B_R,F}(y)dy \le $$
$$(\phi_2*(\frac1R1_{I_R}))(-x_2)\int (\phi_1*(\frac1{R}I_R))(y_1-x_1)w_{B_R,F}(y)dy\lesssim$$
$$\frac1{R}(1+\frac{|x_2|}{R})^{-E}\int (\phi_1*(\frac1{R}I_R))(y_1-x_1)w_{B_R,F}(y)dy.$$
Recalling \eqref{ryyy878t0934=0=4-56k90u=-56i90-870i=8}, we are left with proving that
\begin{equation}
\label{ryyy878t0934=0=4-56k90u=-56i90-870i=85}
\int (\phi_1*(\frac1{R}I_R))(y_1-x_1)w_{B_R,F}(y)dy\lesssim R(1+\frac{|x_1|}{R})^{-E}.
\end{equation}
We split the analysis into three cases. We will need $F\ge 2E+2$.
\bigskip

$(a)\;\;|y_2|\le R.$ In this case
$$\phi_1(x_1)\lesssim \frac1{R^{1/2}}(1+\frac{|x_1|}{R^{1/2}})^{-E}.$$Using \eqref{ryyy878t0934=0=4-56k90u=-56i90-870i=81} with $n=1$ twice (first $R'=R^{1/2}$ then $R'=R$) and \eqref{ryyy878t0934=0=4-56k90u=-56i90-870i=82} with $E_1=E, E_2=2$ we get
$$\int_{|y_2|\lesssim R} (\phi_1*(\frac1{R}I_R))(y_1-x_1)w_{B_R,F}(y)dy\lesssim $$$$\int \frac1R(1+\frac{|x_1-y_1|}{R})^{-E}(1+\frac{|y_1|}{R})^{-E}dy_1\int(1+\frac{|y_2|}{R})^{-2}dy_2\lesssim $$$$ R(1+\frac{|x_1|}{R})^{-E},$$
as needed.

\bigskip

$(b)\;\;|y_2|\sim KR$, with $K\in [1,R^{1/2}]\cap 2^{\N}$. In this case
$$\phi_1(x_1)\lesssim \frac1{R^{1/2}}(1+\frac{|x_1|}{KR^{1/2}})^{-E}, $$and using \eqref{ryyy878t0934=0=4-56k90u=-56i90-870i=81} twice (first $R'=KR^{1/2}$ then $R'=R$) and \eqref{ryyy878t0934=0=4-56k90u=-56i90-870i=82} with $E_1=E, E_2=3$ we write
$$\int_{|y_2|\sim KR} (\phi_1*(\frac1{R}I_R))(y_1-x_1)w_{B_R,F}(y)dy\lesssim $$$$K\int \frac1R(1+\frac{|x_1-y_1|}{R})^{-E}(1+\frac{|y_1|}{R})^{-E}dy_1\int_{|y_2|\sim KR}(1+\frac{|y_2|}{R})^{-3}dy_2\lesssim $$$$ K(1+\frac{|x_1|}{R})^{-E}\frac{KR}{K^3}=\frac{R}{K}(1+\frac{|x_1|}{R})^{-E}.$$
Note that summing over $K\in [1,R^{1/2}]\cap 2^{\N}$ leads to the desired estimate \eqref{ryyy878t0934=0=4-56k90u=-56i90-870i=85}.

\bigskip

$(c)\;\;|y_2|\sim KR^{3/2}$, with $K\in [1,\infty)\cap 2^{\N}$. In this case
$$\phi_1(x_1)\lesssim \frac1{R^{1/2}}(1+\frac{|x_1|}{KR})^{-E}, $$
and so, by \eqref{ryyy878t0934=0=4-56k90u=-56i90-870i=81} we have
$$(\phi_1*(\frac1{R}I_R))(y_1-x_1)\lesssim \frac1{R^{1/2}}(1+\frac{|y_1-x_1|}{KR})^{-E}.$$
Next, combining this with  \eqref{ryyy878t0934=0=4-56k90u=-56i90-870i=82} ($E_1=E$, $E_2=E+2$) and then with \eqref{ryyy878t0934=0=4-56k90u=-56i90-870i=81} we get
$$\int_{|y_2|\sim KR^{3/2}} (\phi_1*(\frac1{R}I_R))(y_1-x_1)w_{B_R,F}(y)dy\lesssim $$
$$R^{1/2}\int (1+\frac{|y_1-x_1|}{KR})^{-E}\frac1R(1+\frac{|y_1|}{R})^{-E}dy_1\int_{|y_2|\sim KR^{3/2}}(1+\frac{|y_2|}{R})^{-E-2}dy_2\lesssim$$
$$R^{1/2}(1+\frac{|x_1|}{KR})^{-E}\frac{KR^{3/2}}{(KR^{1/2})^{E+2}}\lesssim R^{1/2}(1+\frac{|x_1|}{KR})^{-E}\frac{KR^{3/2}}{K^{E+2}R}\le$$$$ \frac{R}{K}(1+\frac{|x_1|}{R})^{-E}.$$
Note that summing over $K\in [1,\infty)\cap 2^{\N}$ leads to the desired estimate \eqref{ryyy878t0934=0=4-56k90u=-56i90-870i=85}.

\end{proof}

\section{$L^2$ decoupling}
\label{L2}
We will use Lemma \ref{nl9} to prove a very simple but efficient decoupling.
This exploits $L^2$ orthogonality and will allow us to decouple to the smallest possible scale, equal to the inverse of the radius of the cube. This process is illustrated by the following simple result.
\begin{proposition}[$L^2$ \textbf{decoupling}]
\label{nl3}
Let  $Q$ be a cube with $l(Q)\ge R^{-1}$. Then for each cube $B_R\subset\R^n$ with side length $R$ we have
$$\|E_{Q}g\|_{L^2(w_{B_R})} \lesssim (\sum_{q \in\Di_{\frac1R}(Q)}\|E_{q}g\|_{L^2(w_{B_R})}^2)^{1/2}.$$
\end{proposition}

\begin{proof}
We will prove  that
\begin{equation}
\label{ne89}
\|E_{Q}g\|_{L^2(w_{B_R})}^2\lesssim \sum_{q\in\Di_{\frac1R}(Q)}\|E_{q}g\|_{L^2(w_{B_R})}^2.
\end{equation}
Fix   a positive Schwartz function $\eta$ such that the Fourier transform of $\sqrt{\eta}$ is supported in a small neighborhood of the origin, and such that $\eta\ge 1$ on $B(0,1)$.
By invoking Lemma \ref{nl9} we see that inequality \eqref{ne89} will follow once we check that
\begin{equation}
\label{ne90}
\|E_{Q}g\|_{L^2({B'})}^2\lesssim \sum_{q\in\Di_{\frac1R}(Q)}\|E_{q}g\|_{L^2(\eta_{B'})}^2
\end{equation}
holds true for each cube $B'$ with  $l(B')=R$.

Note that the Fourier transform of $\sqrt{\eta_{B'}}E_{q}g$ will be supported inside the $R^{-1}-$neighborhood of the paraboloid above $q$, and that these neighborhoods are pairwise disjoint for two non adjacent $q$.
Since
$$\|E_{Q}g\|_{L^2({B'})}^2\lesssim \|E_{Q}g\|_{L^2(\eta_{B'})}^2=\|\sqrt{\eta_{B'}}E_{Q}g\|_{L^2(\R^n)}^2,$$
\eqref{ne90} will now immediately follow from the $L^2$ orthogonality of the functions $\sqrt{\eta_{B'}}E_{q}g$.

\end{proof}

\section{Parabolic rescaling}\label{s5}
\bigskip

A nice property of the paraboloid $\W^{n-1}$ is the fact that each square-like cap on it can be stretched to the whole $\W^{n-1}$ via an affine transformation. Affine transformations interact well with the Fourier transform, and this facilitates a natural passage from the operator $E_Q$ to $E_{[0,1]^{n-1}}$.

\begin{proposition}
\label{propo:parabooorescal}
Let $0<\delta\le \sigma<1$ and $p\ge 2$. For each  cube $Q\subset [0,1]^{n-1}$ with $l(Q)=\sigma^{\frac12}$ and each cube $B\subset \R^n$ with $l(B)\ge \delta^{-1}$ we have
$$\|E_{Q}g\|_{L^p(w_B)}\lesssim \D_p(\frac\delta\sigma)(\sum_{q\in\Di_{\delta^{1/2}}(Q)}\|E_qg\|_{L^p(w_B)}^2)^{1/2}.$$
The implicit constant is independent of $\delta,\sigma,Q,B$.
\end{proposition}
\begin{proof}
Let us first assume $l(B)= \delta^{-1}$.
We will apply Lemma \ref{nl9} to
$$O_1(v)=\|E_{Q}g\|_{L^p(v)}^p$$
$$O_2(v)=(\sum_{q\in\Di_{\delta^{1/2}}(Q)}\|E_qg\|_{L^p(v)}^2)^{\frac{p}{2}},$$
cf. Remark \ref{aha}.
It thus suffices to prove that
$$\|E_{Q}g\|_{L^p(B)}\lesssim \D_p(\frac\delta\sigma)(\sum_{q\in\Di_{\delta^{1/2}}(Q)}\|E_qg\|_{L^p(w_B)}^2)^{1/2}.$$
Assume $Q=a+[0,\sigma^{1/2}]^{n-1}$ with $a=(a_1,\ldots,a_{n-1})$. We will perform a parabolic rescaling via the affine transformation $L=L_Q$
$$L_Q(\xi_1,\ldots,\xi_{n-1})=(\xi_1',\ldots,\xi_{n-1}')=(\frac{\xi_1-a_1}{\sigma^{1/2}},\ldots,\frac{\xi_{n-1}-a_{n-1}}{\sigma^{1/2}} ).$$
A simple computation shows that for each cube $\tilde{Q}$ we have
$$|E_{\tilde{Q}}g(x_1,\ldots,x_{n-1},x_n)|=\sigma^{\frac{n-1}{2}}|E_{\tilde{Q}_L}g_L((x_1+2a_1x_n)\sigma^{1/2},\ldots,(x_{n-1}+2a_{n-1}x_n)\sigma^{1/2},x_n\sigma )|,$$
where  $\tilde{Q}_L=L(\tilde{Q})$, $g_L=g\circ L$. The image $S$ of $B$ under the affine transformation $T=T_Q$
$$T_Q(x_1,\ldots,x_{n-1},x_n)=((x_1+2a_1x_n)\sigma^{1/2},\ldots,(x_{n-1}+2a_{n-1}x_n)\sigma^{1/2},x_n\sigma )$$
can be covered with a family $\F$ of pairwise disjoint cubes $\Delta$ with side length $\delta^{-1}\sigma$,  such that we have the following double inequality, in the same spirit as \eqref{jgiou tiguyt reym89r--9048-0023=-1r90-8=4-293=-0}
\begin{equation}
\label{iodferuigyiurtuygrtiohhutoi}
1_S(x)\lesssim \sum_{B'\in\F}w_{B'}(x)\lesssim w_B(T^{-1}x).
\end{equation}
The second inequality is very easy to guarantee for a proper covering, as $l(B')\le l(B)$.
After a change of variables on the spatial side we get (since $Q_L=[0,1]^{n-1}$)
$$\|E_{Q}g\|_{L^p(B)}=\sigma^{\frac{n-1}{2}}\sigma^{-\frac{n+1}{2p}}\|Eg_L\|_{L^p(S)}\lesssim$$
$$\sigma^{\frac{n-1}{2}-\frac{n+1}{2p}}(\sum_{B'\in\F}\|Eg_L\|_{L^p(w_{B'})}^p)^{1/p}\le$$
$$\sigma^{\frac{n-1}{2}-\frac{n+1}{2p}}\D_p(\frac\delta\sigma)[\sum_{B'\in\F}(\sum_{q'\in\Di_{(\frac{\delta}{\sigma})^{1/2}}([0,1]^{n-1})}\|E_{q'}g_L\|_{L^p(w_{B'})}^2)^{\frac{p}{2}}]^{1/p}=$$
$$\sigma^{\frac{n-1}{2}-\frac{n+1}{2p}}\D_p(\frac\delta\sigma)[\sum_{B'\in\F}(\sum_{q\in\Di_{\delta^{1/2}}(Q)}\|E_{q_L}g_L\|_{L^p(w_{B'})}^2)^{\frac{p}{2}}]^{1/p}.$$
Using Minkowski's inequality followed by \eqref{iodferuigyiurtuygrtiohhutoi}, this is dominated by
$$\sigma^{\frac{n-1}{2}-\frac{n+1}{2p}}\D_p(\frac\delta\sigma)(\sum_{q\in\Di_{\delta^{1/2}}(Q)}\|E_{q_L}g_L\|_{L^p(\sum w_{B'})}^2)^{\frac{1}{2}}\lesssim$$
$$\sigma^{\frac{n-1}{2}-\frac{n+1}{2p}}\D_p(\frac\delta\sigma)(\sum_{q\in\Di_{\delta^{1/2}}(Q)}\|E_{q_L}g_L\|_{L^p(w_B\circ T^{-1})}^2)^{\frac{1}{2}}.$$
By changing back to the original variables, this is easily seen to be the same as
$$\D_p(\frac\delta\sigma)(\sum_{q\in\Di_{\delta^{1/2}}(Q)}\|E_qg\|_{L^p(w_B)}^2)^{1/2}.$$
This finishes the proof in the case $l(B)= \delta^{-1}$.
\medskip

Let us next assume $l(B)\ge \delta^{-1}$. By invoking again Lemma \ref{nl9} (cf. Remark \ref{aha}), it suffices to prove
$$\|E_{Q}g\|_{L^p(B)}\lesssim \D_p(\frac\delta\sigma)(\sum_{q\in\Di_{\delta^{1/2}}(Q)}\|E_qg\|_{L^p(w_B)}^2)^{1/2}.$$

Using \eqref{jgiou tiguyt reym89r--9048-0023=-1r90-8=4-293=-0} and Minkowski's inequality, we may close the argument as follows
$$\|E_{Q}g\|_{L^p(B)}\lesssim (\sum_{\Delta\in\Di_{\delta^{-1}}(B)}\|E_{Q}g\|_{L^p(w_\Delta)}^p)^{1/p}$$
$$\lesssim \D_p(\frac\delta\sigma)(\sum_{\Delta\in\Di_{\delta^{-1}}(B)}(\sum_{q\in\Di_{\delta^{1/2}}(Q)}\|E_qg\|_{L^p(w_\Delta)}^2)^{p/2})^{1/p}.$$
$$\lesssim \D_p(\frac\delta\sigma)(\sum_{q\in\Di_{\delta^{1/2}}(Q)}\|E_qg\|_{L^p(w_B)}^2)^{1/2}.$$

\end{proof}

\section{Linear versus multilinear decoupling}

\label{s10}
Let $\pi:\W^{n-1}\to [0,1]^{n-1}$ be the projection map.
\begin{definition}
We say that the cubes  $Q_1,\ldots,Q_n\subset [0,1]^{n-1}$ are $\nu$-transverse if the volume of the parallelepiped spanned by unit normals $n(P_i)$  is greater than $\nu$, for each choice of $P_i\in \W^{n-1}$ with $\pi(P_i)\in Q_i$.
\end{definition}
For $E\ge 100n$, $2\le p\le \infty$, $m\in \N$ and $0<\nu<1$ we let $\D(\delta,p,\nu,m,E)=\D_n(\delta,p,\nu,m,E)$ be the smallest constant such that the inequality
\begin{equation}
\label{ nbo wru45y9d0=1-=c0214=-`=2e4=23-4}
[\sum_{\Delta\in\Di_{\mu^{-1}}(B)}(\prod_{i=1}^n\|E_{Q_i}g\|_{L^p(w_{\Delta,10E})}^p)^{\frac{1}n}]^{\frac1p}\le
\end{equation}
$$
\D(\delta,p,\nu,m,E)\left[\prod_{i=1}^n\sum_{q_i\in\Di_{\delta^{1/2}}(Q_i)}\|E_{q_i}g\|_{L^p(w_{B,E})}^2\right]^{\frac1{2n}}$$
holds for each cube $B\subset \R^n$ with  $l(B)=\delta^{-1}$, each $g:[0,1]^{n-1}\to\C$ and for each $\nu$-transverse  cubes $Q_i$ with equal side lengths $\mu$ satisfying $\mu\ge \delta^{2^{-m}}$.  Recall that  $\Di_{\mu^{-1}}(B)$ is the partition of $B$ using cubes $\Delta$ with $l(\Delta)=\mu^{-1}$.
The lower bound we impose on the size of $\mu$ is a bit more severe than the minimal lower bound $\mu\ge \delta^{1/2}$ needed in order to make sense of the quantity $\Di_{\delta^{1/2}}(Q_i)$. This restriction can be ignored for now and should only be paid attention to in  the final argument from the last section, when dominating \eqref{ fu9ety8290-78-=1=23 -r02v9=--e322} by \eqref{ fu9ety8290-78-=1=23 -r02v9=--e32}
\smallskip

Note that we use $w_{\Delta,10E}$ rather than $w_{\Delta,E}$ in \eqref{ nbo wru45y9d0=1-=c0214=-`=2e4=23-4}. This is done for purely technical reasons, as explained in Remark \ref{aha1}.

\medskip

Since $|E_{Q_i}g|$ can be thought of as being essentially constant on each $\Delta$, the quantity
$$[\sum_{\Delta\in\Di_{\mu^{-1}}(B)}(\prod_{i=1}^n\|E_{Q_i}g\|_{L^p(w_{\Delta,10E})}^p)^{\frac{1}n}]^{\frac1p}$$
can be viewed as being comparable to
$$\||\prod_{i=1}^{n}E_{Q_i}g|^{\frac1n}\|_{L^p(w_{B,10E})}.$$
The former will be a preferred substitute for the latter due to purely technical reasons.

\bigskip

Several applications of  H\"older's inequality combined with \eqref{jgiou tiguyt reym89r--9048-0023=-1r90-8=4-293=-0} show that for each $\nu$, $m$
\begin{equation}
\label{hjgvhdgfyrtefuiyecuryueuyguirtygio}
\D(\delta,p,\nu,m,E)\lesssim \D(\delta,p,E).
\end{equation}
This inequality is too basic and will never be used. We will instead derive a stronger form of it in the last section, see \eqref{dschjgdshdcyfyuvtioduewuyt}, which dominates $\D(\delta,p,\nu,m,E)$ using a combination of powers of $\delta$ and some power of $\D(\delta,p)$.
\medskip

We will now prove and later use the following  approximate reverse inequality. Recall the definition of $\Gamma_n(E)$ from Theorem \ref{alternative}.
\begin{theorem}
\label{ch ft7wbtfgb6n17r782brym9,iqmivpk[l}
Let $E\ge 100n$.
Assume one of the following holds

(i) $n=2$

(ii) $n\ge 3$ and $\D_{n-1}(\delta,p,\Gamma_{n-1}(10E))\lesssim_{\epsilon,E} \delta^{-\epsilon}$.
\\
Then for each $0<\nu\le 1$ there is $\epsilon(\nu)=\epsilon(\nu,p, E)$ with $\lim_{\nu\to 0}\epsilon (\nu)=0$ and $C_{\nu,m}$ such that for each $m\ge 1$ we have
\begin{equation}\label{dbhcgvbdduiqwyfyueeeoiowuydrtuiwqo}
\D_n(R^{-1},p,E)\le C_{\nu,m} R^{\epsilon(\nu)}(1+\sup_{1\le R'\le  R}\D_n({R'}^{-1},p,\nu,m,E)),
\end{equation}
for each $R\gtrsim_{\nu,m}1$.
\end{theorem}
\medskip

\medskip

We next prove the case $n=3$ of the theorem and will then indicate the modifications needed for $n\ge 4$. The argument will also show how to deal with the case $n=2$.

\begin{remark}
\label{markre1dd5}
If $P_1,P_2,P_3\in \W^{2}$, the volume of the parallelepiped spanned by the  unit normals $n(P_i)$ is comparable to the area of the triangle with vertices $\pi(P_i)$.
\end{remark}

The key step in the proof of Theorem \ref{ch ft7wbtfgb6n17r782brym9,iqmivpk[l} for $n=3$ is the following result.

\begin{proposition}
\label{hcnyf7yt75ycn8u32r8907n580-9=--qc mvntvu5n8t}
 Assume $\D_2(\delta,p,\Gamma_2(10E))\lesssim_\epsilon {\delta}^{-\epsilon}$. Then there are constants $C, C_\epsilon$ such that for each  $m\ge 1$ and each $R\ge K^{2^m}$
$$\|Eg\|_{L^p(w_{B_R,E})}\le C_\epsilon K^{\epsilon}[(\sum_{\alpha\in\Di_{K^{-1}}([0,1]^2)}\|E_\alpha g\|_{L^p(w_{B_R,E)}}^2)^{1/2}+(\sum_{\beta\in\Di_{K^{-1/2}}([0,1]^2)}\|E_\beta g\|_{L^p(w_{B_R,E})}^2)^{1/2}]$$$$+K^{C}\D_3(R^{-1},p, K^{-2},m,E)(\sum_{\Delta\in\Di_{R^{-1/2}}([0,1]^2)}\|E_{\Delta}g\|_{L^p(w_{B_R,E})}^2)^{1/2}.$$
\end{proposition}
\begin{proof}
Using Lemma \ref{nl9} (cf. Remark \ref{aha}), it suffices to prove the inequality with the unweighted quantity $\|Eg\|_{L^p({B_R})}$ on the left hand side.
Cover $B_R$ with a family $\Di_K(B_R)$ of  cubes $B_K\subset \R^3$ with side length $K$.

Let $\psi:\R^3\to\C$ be a Schwartz function with Fourier transform equal to 1 on $B(0,10)$.
For each $\alpha\in \Di_{K^{-1}}([0,1]^2)$ define
$$c_\alpha(B_K)=(\frac1{|B_K|}\int |E_\alpha g|^pw_{B_K,10E})^{\frac1p}.$$
Note that since $E_\alpha g=(E_\alpha g)*\psi_{K}$ for an appropriate modulation/dilation $\psi_K$ of $\psi$, we have
$$\sup_{x\in B_K}|E_\alpha g(x)|\lesssim c_\alpha(B_K).$$

This is a manifestation of the Uncertainty Principle that asserts that $|E_\alpha g|$ is essentially constant at scale $K$.
Let $\alpha^*=\alpha^*(K)\in \Di_{K^{-1}}([0,1]^2)$ be a square that maximizes $c_\alpha(B_K)$.
Define also
$$S_{big}=\{\alpha: c_{\alpha}(B_K)\ge K^{-C}c_{\alpha^*}(B_K)\}.$$
The number $C$ will change its value from one line to the next one, but crucially, it will always be independent of $K$.

We will show that for each $B_K\in \Di_K(B_R)$ there exists a line $L=L(B_K)$ in the $(\xi_1,\xi_2)$ plane such that if
$$S_L=\{(\xi_1,\xi_2): \dist((\xi_1,\xi_2),L)\le \frac{C}{K}\}$$
then for $x\in B_K$
$$ |Eg(x)|\le $$
\begin{equation}
\label{term1.1}Cc_{\alpha^*}(B_K)+
\end{equation}\begin{equation}
\label{term1.2}
K^{4}\max_{\alpha_1,\alpha_2,\alpha_3\atop{K^{-2}-\text{transverse}}}(\prod_{i=1}^3c_{\alpha_i}(B_K))^{1/3}+
\end{equation}\begin{equation}
\label{term1.3}
|\sum_{\alpha\subset S_L}E_\alpha g(x)|.
\end{equation}
To see this, we distinguish three scenarios. First, if there is no  $\alpha\in S_{big}$ with  $\dist(\alpha,\alpha^*)\ge \frac{10}{K}$, then \eqref{term1.1} suffices, as
$$ |Eg(x)|\le \sum_{\alpha}c_\alpha(B_K)\le Cc_{\alpha^{*}}(B_K).$$
Otherwise, there is  $\alpha^{**}\in S_{big}$ with $\dist(\alpha^{**},\alpha^*)\ge \frac{10}{K}$.
The line $L$ is determined by the centers of $\alpha^1,\alpha^{2}$, which are chosen to be  furthest apart among
all possible pairs in $S_{big}$. Note that the distance between these centers is at least $\frac{10}{K}$.
\bigskip

Second, if there is $\alpha^3\in S_{big}$ such that  $\alpha^3$ intersects the complement of $S_L$  then  \eqref{term1.2} suffices. To see this, note first that $\alpha^3$ is forced to intersect the strip between $\alpha^1$ and $\alpha^2$ perpendicular to $L$. Thus, a triangle determined by any three points in $\alpha^i$ has area $\gg K^{-2}$. Combining this with  Remark \ref{markre1dd5} shows  that $\alpha^{1},\alpha^{2}$, $\alpha^{3}$ are $K^{-2}$ transverse, for $C$ large enough.
\bigskip

Third, if all $\alpha\in S_{big}$ are inside $S_L$,  the sum of \eqref{term1.1} and \eqref{term1.3} will obviously suffice.
\\
We now claim that \eqref{term1.1}-\eqref{term1.3} imply the following
$$\|Eg\|_{L^p(B_K)}$$$$\le C_\epsilon K^{\epsilon}[ (\sum_{\alpha\in\Di_{K^{-1}}([0,1]^2)}\|E_\alpha g\|_{L^p(w_{B_K,10E})}^2)^{1/2}+(\sum_{\beta\in\Di_{K^{-1/2}}([0,1]^2)}\|E_\beta g\|_{L^p(w_{B_K,10E})}^2)^{1/2}]$$
\begin{equation}
\label{uyfuyr7fgyrtyh0ytoujp[io[p]l[p}
+K^{C}\max_{\alpha_1,\alpha_2,\alpha_3\atop{K^{-2}-\text{transverse}}}(\prod_{i=1}^3\|E_{\alpha_i}g\|_{L^p(w_{B_K,10E})})^{1/3}.
\end{equation}
Only the third scenario above needs an explanation.   Cover $S_L$ by pairwise disjoint rectangles $U$ of dimensions $K^{-1}$ and ${K^{-1/2}}$, with the long side parallel to $L$. To simplify notation, assume the equation of $L$ is $\eta=1$ and that $B_K=[0,K]^3$. For each fixed $y$ the Fourier transform of $(x,z)\mapsto E_{S_L}g(x,y,z)$ is supported in the $O(K^{-1})$ neighborhood of the parabola $\eta=\xi^2+1$. Using Theorem \ref{alternative} and our hypothesis $\D_{2}(K^{-1},p,\Gamma_2(10E))\lesssim_\epsilon K^{\epsilon}$ we can write
\begin{equation}
\label{huregyerwiout8rtut9u56uy956iy89790}
\|E_{S_L} g(x,y,z)\|_{L^p_{x,z}([0,K]^2)}\lesssim \D_{2}(K^{-1},p,10E)(\sum_{U}\|E_{U}g(x,y,z)\|_{L^p_{x,z}(w_{[0,K]^2,10E})}^2)^{1/2}
\end{equation}
$$\lesssim_\epsilon\delta^{-\epsilon}(\sum_{U}\|E_{U}g(x,y,z)\|_{L^p_{x,z}(w_{[0,K]^2,10E})}^2)^{1/2}.$$

Next raise this inequality to the power $p$, integrate over $y\in [0,K]$ and use
\begin{equation}
\label{huregyerwiout8rtut9u56uy956iy89790rfguity65y956908yu9076}
w_{[0,K]^2,10E}(x,z)1_{[0,K]}(y)\lesssim w_{B_K,10E}(x,y,z)
\end{equation}
and
Minkowski's inequality to write
$$\|\sum_{\alpha:\alpha\subset S_L}E_\alpha g\|_{L^p({B_K})}\lesssim_\epsilon K^{\epsilon}(\sum_{U}\|E_{U}g\|_{L^p(w_{B_K,10E})}^2)^{1/2}.$$
Note however that since we are dealing with the third scenario,  the contribution of $E_{[0,1]^2\setminus S_L}g$ is small
$$\|E_{[0,1]^2\setminus S_L}g\|_{L^p(w_{B_K,10E})}\le \sum_{\alpha\not\in S_{big}}\|E_{\alpha}g\|_{L^p(w_{B_K,10E})}\le Cc_{\alpha^*}(B_K)|B_K|^{1/p}.$$

Using the triangle inequality we get
$$(\sum_{U}\|E_{U}g\|_{L^p(w_{B_K,10E})}^2)^{1/2}\le (\sum_{\beta\in\Di_{K^{-1/2}}([0,1]^2)}\|E_\beta g\|_{L^p(w_{B_K,10E})}^2)^{1/2}+C\|E_{\alpha^*}g\|_{L^p(w_{B_K,10E})}.$$
We conclude that \eqref{uyfuyr7fgyrtyh0ytoujp[io[p]l[p} holds under the third scenario. The first two scenarios are quite immediate.

Using $w_{B_K,10E}\le w_{B_K,E}$, \eqref{uyfuyr7fgyrtyh0ytoujp[io[p]l[p} further implies that
$$\|Eg\|_{L^p(B_K)}$$$$\le C_\epsilon K^{\epsilon}[ (\sum_{\alpha\in\Di_{K^{-1}}([0,1]^2)}\|E_\alpha g\|_{L^p(w_{B_K,E})}^2)^{1/2}+(\sum_{\beta\in\Di_{K^{-1/2}}([0,1]^2)}\|E_\beta g\|_{L^p(w_{B_K,E})}^2)^{1/2}]$$
\begin{equation}
\label{uyfuyr7fgyrtyh0ytoujp[io[p]l[p111}
+K^{C}\max_{\alpha_1,\alpha_2,\alpha_3\atop{K^{-2}-\text{transverse}}}(\prod_{i=1}^3\|E_{\alpha_i}g\|_{L^p(w_{B_K,10E})})^{1/3}.
\end{equation}
Finally, we  raise \eqref{uyfuyr7fgyrtyh0ytoujp[io[p]l[p111} to the power $p$ and sum over all $B_K\in\Di_K(B_R)$, invoking Minkowski's inequality and \eqref{jgiou tiguyt reym89r--9048-0023=-1r90-8=4-293=-0} to get
$$\|Eg\|_{L^p({B_R})}\le $$$$C_\epsilon K^{\epsilon}[(\sum_{\alpha\in\Di_{K^{-1}}([0,1]^2)}\|E_\alpha g\|_{L^p(w_{B_R,E)}}^2)^{1/2}+(\sum_{\beta\in\Di_{K^{-1/2}}([0,1]^2)}\|E_\beta g\|_{L^p(w_{B_R,E})}^2)^{1/2}]$$$$+K^{C}\D_3(R^{-1},p, K^{-2},m,E)(\sum_{\Delta\in\Di_{R^{-1/2}}([0,1]^2)}\|E_{\Delta}g\|_{L^p(w_{B_R,E})}^2)^{1/2}.$$
An application of Lemma \ref{nl9} finishes the proof.
\end{proof}

Parabolic rescaling as in the proof of Proposition \ref{propo:parabooorescal} leads to the following. The details are left to the reader.

\begin{proposition}
\label{jcfn vrfwyt8981  =,i90t8-kc0r=90,it90}
Let $\tau\subset [0,1]^2$ be a square with side length $\delta\ge R^{-1/2}K^{2^{m-1}}$. Assume $$\D_2(\delta',p,\Gamma_2(10E))\lesssim_\epsilon {\delta'}^{-\epsilon}$$
for  all $\delta'<1$. Then if $R\ge K^{2^m}$ we have
$$\|E_{\tau}g\|_{L^p(w_{B_R,E})}\le C_\epsilon K^{\epsilon}[(\sum_{\alpha\in\Di_{\delta K^{-1}}(\tau)}\|E_\alpha g\|_{L^p(w_{B_R,E)}}^2)^{1/2}+(\sum_{\beta\in\Di_{\delta K^{-1/2}}(\tau)}\|E_\beta g\|_{L^p(w_{B_R,E})}^2)^{1/2}]$$$$+K^{C}\D_3((R\delta^2)^{-1},p, K^{-2},m,E)(\sum_{\Delta\in\Di_{R^{-1/2}}(\tau)}\|E_{\Delta}g\|_{L^p(w_{B_R,E})}^2)^{1/2}.$$
The constants $C_\epsilon,C$ are independent of $\delta,R,\tau,K$.
\end{proposition}
\bigskip

We are now ready to prove Theorem \ref{ch ft7wbtfgb6n17r782brym9,iqmivpk[l} for $n=3$. Let $K=\nu^{-1/2}$. Let also $R\ge K^{2^m}=\nu^{-2^{m-1}}$. Iterate Proposition \ref{jcfn vrfwyt8981  =,i90t8-kc0r=90,it90} starting with scale $\delta=1$ until we reach scale $\delta=R^{-1/2}K^{2^{m-1}}$.
Each iteration  lowers the scale of the square from $\delta$ to at least $\frac{\delta}{K^{1/2}}$. Thus we have to iterate $O(\log_K R)$ times. We use the following immediate consequence of H\"older's inequality
\begin{equation}
\label{ycvy7vr62c906787-`=xiz=907-r89763489--=2}
(\sum_{\beta\in\Di_{R^{-1/2}K^{2^{m-1}}}([0,1]^2)}\|E_{\beta}g\|_{L^p(w_{B_R,E})}^2)^{1/2}\lesssim K^{O(1)}(\sum_{\Delta\in\Di_{R^{-1/2}}([0,1]^2)}\|E_{\Delta}g\|_{L^p(w_{B_R,E})}^2)^{1/2}.
\end{equation}
Since $$\D_3((\delta^2R)^{-1},p,\nu,m,E)\le\sup_{1\le R'\le  R}\D_3({R'}^{-1},p,\nu,m,E)$$ we get
$$\|E_{[0,1]^2}g\|_{L^p(w_{B_R,E})}\le $$$$(CC_\epsilon K^\epsilon)^{O(\log_KR)} K^{O(1)}(1+\sup_{1\le R'\le  R}\D_3({R'}^{-1},p,\nu,m,E))(\sum_{\Delta\in\Di_{R^{-1/2}}([0,1]^2)}\|E_{\Delta}g\|_{L^p(w_{B_R,E})}^2)^{1/2}$$
$$=R^{-O(1)\log_{\nu}(CC_\epsilon)+\epsilon}\nu^{-O(1)}(1+\sup_{1\le R'\le  R}\D_3({R'}^{-1},p,\nu,m,E))(\sum_{\Delta\in\Di_{R^{-1/2}}([0,1]^2)}\|E_{\Delta}g\|_{L^p(w_{B_R,E})}^2)^{1/2}.$$
The result in Theorem \ref{ch ft7wbtfgb6n17r782brym9,iqmivpk[l} now follows since $C,C_\epsilon$ do not depend on $\nu$.
\bigskip

To summarize, the proof of Theorem \ref{ch ft7wbtfgb6n17r782brym9,iqmivpk[l}  for $n=3$ relied on the hypothesis that the contribution coming from squares $\beta$ living near a line is controlled by the negligible lower dimensional quantity $\D_{2}(\delta,p)=O(\delta^{-\epsilon})$. When $n\ge 4$, the contribution from the cubes near a hyperplane $H$ in $[0,1]^{n-1}$ will be similarly controlled by $\D_{n-1}(\delta,p)$. That is because $\pi^{-1}(H)$ is a lower dimensional  elliptic paraboloid whose principal curvatures are $\sim 1$, uniformly over $H$. This paraboloid is an affine image of $\W^{n-2}$, and can be analyzed using parabolic rescaling. When $n=2$, there is no such lower dimensional contribution.

\bigskip

\section{From multilinear Kakeya to multilinear decouplings}
\bigskip

We start by recalling the following multilinear Kakeya inequality due to Bennett, Carbery and Tao, \cite{BCT}. We refer the reader to \cite{Gu} for a  short proof.
\begin{theorem}
\label{nt3}
Let $0<\nu<1$. Consider $n$ families $\P_j$ consisting of tiles (rectangular boxes) $P$ in $\R^n$ having the following properties

(i)  each $P$ has  $n-1$ side lengths  equal to $R^{1/2}$  and one side length equal to $R$ which points in the direction of the unit vector  $v_P$
\medskip

(ii) $v_{P_1}\wedge\ldots\wedge v_{P_n}\ge \nu$ for each $P_i\in\P_i$
\medskip

(iii) all tiles are subsets of a fixed cube  $B_{4R}$ of side length $4R$
\medskip

Then we have the following inequality
\begin{equation}
\label{ne3}
\nint_{B_{4R}}|\prod_{j=1}^{n}F_j|^{\frac{1}{n-1}}\lesssim_{\epsilon, \nu }R^{\epsilon} \prod_{j=1}^{n}|\nint_{B_{4R}}F_j|^{\frac1{n-1}}
\end{equation}
for all functions  $F_j$ of the form
$$F_j=\sum_{P\in\P_j}c_P1_P.$$
The implicit constant will not depend on $R,c_P,\P_j$.
\end{theorem}

We use this to prove the following key result.

\begin{theorem}
\label{nt4}
Let  $p\ge \frac{2n}{n-1}$ and $\delta<1$. Consider $n$ $\nu$-transverse cubes $Q_1,\ldots,Q_n\subset[0,1]^{n-1}$.
Let $B$ be an arbitrary cube in $\R^n$ with side length $\delta^{-2}$, and let $\B$ be the (unique) partition  of $B$ into  cubes $\Delta$ of side length  $\delta^{-1}$. Then for each $g:[0,1]\to \C$ we have
\begin{equation}
\label{hdgfhgdsfdshfjhdkghfdgkjfkjhgklfirs}
\frac1{|\B|}\sum_{\Delta\in\B}\left[\prod_{i=1}^{n}(\sum_{Q_{i,1}\in\Di_{\delta}(Q_i)}\|E_{Q_{i,1}}g\|_{L_{\sharp}^{\frac{p(n-1)}{n}}(w_\Delta)}^2)^{1/2}\right]^{p/{n}}
\end{equation}
\begin{equation}
\label{hdgfhgdsfdshfjhdkghfdgkjfkjhgkl}
\lesssim_{\epsilon,\nu}\delta^{-\epsilon} \left[\prod_{i=1}^{n}(\sum_{Q_{i,1}\in\Di_{\delta}(Q_{i})}\|E_{Q_{i,1}}g\|_{L_{\sharp}^{\frac{p(n-1)}{n}}(w_B)}^2)^{1/2}\right]^{p/n}.
\end{equation}
Moreover, the implicit constant is independent of $g,\delta,B$.
\end{theorem}
\begin{remark}
This result is part of a two-stage process. Note that, strictly speaking, this inequality is not a decoupling, since the size on the frequency cubes $Q_{i,1}$  remains unchanged. However, the size of the spatial cube increases from $\delta^{-1}$ to $\delta^{-2}$, which will facilitate a subsequent decoupling, as we shall later see in Proposition \ref{hegfuyefuyrufyuryfuryowqiuoiuwdoiwd9438u9uirhcfj}.

\end{remark}
\begin{proof}
Since we can afford logarithmic losses in $\delta$, it suffices to prove the inequality with the summation on both sides restricted to  families of $Q_{i,1}$ for which  $\|E_{Q_{i,1}}g\|_{L_{\sharp}^{\frac{p(n-1)}{n}}(w_B)}$ have comparable size (within a multiplicative factor of 2), for each $i$. Indeed, the cubes $Q_{i,1}'$ satisfying (for some large enough $C=O(1)$)
$$\|E_{Q_{i,1}'}g\|_{L_{\sharp}^{\frac{p(n-1)}{n}}(w_B)}\le \delta^C\max_{Q_{i,1}\in \Di_{\delta}(Q_i)}\|E_{Q_{i,1}}g\|_{L_{\sharp}^{\frac{p(n-1)}{n}}(w_B)}$$
can be easily dealt with by using the triangle inequality, since we automatically have
$$\max_{\Delta\in \B}\|E_{Q_{i,1}'}g\|_{L_{\sharp}^{\frac{p(n-1)}{n}}(w_\Delta)}\le \delta^C\max_{Q_{i,1}\in \Di_{\delta}(Q_i)}\|E_{Q_{i,1}}g\|_{L_{\sharp}^{\frac{p(n-1)}{n}}(w_B)}.$$
This leaves only $\log_2 (\delta^{-O(1)})$ sizes to consider.
\medskip

Let us now assume that we have $N_i$ cubes $Q_{i,1}$, with $\|E_{Q_{i,1}}g\|_{L_{\sharp}^{\frac{p(n-1)}{n}}(w_B)}$ of comparable size. Since $p\ge \frac{2n}{n-1}$, by H\"older's inequality \eqref{hdgfhgdsfdshfjhdkghfdgkjfkjhgklfirs} is  at most
\begin{equation}
\label{ne4}
(\prod_{i=1}^{n}N_i^{\frac12-\frac{n}{p(n-1)}})^{p/{n}}\frac1{|\B|}\sum_{\Delta\in\B}\left[\prod_{i=1}^{n}(\sum_{Q_{i,1}}\|E_{Q_{i,1}}g\|_{L_{\sharp}^{\frac{p(n-1)}{n}}(w_\Delta)}^{\frac{p(n-1)}{n}})\right]^{\frac{1}{ n-1}}.
\end{equation}

\medskip

For each cube $Q=Q_{i,1}$ centered at $c_Q$ we cover $B$ with a family $\F_Q$ of  pairwise disjoint, mutually parallel tiles $ T_Q$.  They have   $n-1$ short sides of length $\delta^{-1}$ and one longer side of length $\delta^{-2}$, pointing in the direction of the normal $N(c_Q)$ to the paraboloid $\W^{n-1}$ at $c_Q$. Moreover, we can assume these tiles to be inside the cube $4B$. We let $ T_Q(x)$ be the  tile containing $x$, and we let $2  T_Q$ be the dilation of $ T_Q$ by a factor of 2 around its center.

Let us use $q$ to abbreviate $p(n-1)/n$.     Our goal is to control the expression

$$ \frac1{|\B|}\sum_{\Delta\in\B} \prod_i \left( \sum_{Q_{i,1}} \| E_{Q_{i,1}} g \|_{L^q_{\sharp}(w_\Delta)}^q \right)^{\frac1{n-1}}. $$

We now define $F_Q$ for $x\in \cup_{ T_Q\in\F_Q} T_Q$ by

$$ F_Q(x) := \sup_{y \in 2  T_Q(x)} \| E_Q g \|_{L^q_{\sharp}(w_{B(y, \delta^{-1})})}. $$

For any point $x \in \Delta$ we have $ \Delta \subset 2  T_Q(x)$, and so we also have

$$ \| E_Q g \|_{L^q_{\sharp}(w_\Delta)} \le F_Q (x). $$

Therefore,

$$ \frac1{|\B|}\sum_{\Delta \in \B} \prod_i \left( \sum_{Q_{i,1}} \| E_{Q_{i,1}} g \|_{L^q_{\sharp}(w_\Delta)}^q \right)^{\frac1{n-1}} \lesssim \nint_{4B} \prod_i (\sum_{Q_{i,1}} F_{Q_{i,1}}^q )^{\frac1{n-1}}. $$

Moreover, the function $F_Q^q$ is constant on each tile $ T_Q\in\F_Q$.  Applying Theorem  \ref{nt3} we get the bound

$$ \nint_{4B} \prod_{i} (\sum_{Q_{i,1}} F_{Q_{i,1}}^q )^{\frac1{n-1}} \lesssim_{\epsilon,\nu} \delta^{-\epsilon} \prod_i \left( \sum_{Q_{i,1}} \nint_{4B} F_{Q_{i,1}}^q \right)^{\frac1{n-1}}. $$

It remains to check that for each $Q=Q_{i,1}$

\begin{equation} \label{FJbound} \| F_Q \|_{L^q_{\sharp}(4B)} \lesssim \| E_Q g \|_{L^q_{\sharp}(w_B)}. \end{equation}
Once this is established, it follows  that \eqref{ne4} is dominated by
\begin{equation}
\label{ne5}
\delta^{-\epsilon}(\prod_{i=1}^{n}N_i^{\frac12-\frac{n}{p(n-1)}})^{p/{n}}\prod_{i=1}^{n}(\sum_{Q_{i,1}}\|E_{Q_{i,1}}g\|_{L_{\sharp}^{\frac{p(n-1)}{n}}(w_B)}^{\frac{p(n-1)}n})^{\frac{1}{n-1}}.
\end{equation}
Recalling the restriction we have made on $Q_{i,1}$, \eqref{ne5} is comparable to
$$\delta^{-\epsilon}\left[\prod_{i=1}^{n}(\sum_{Q_{i,1}}\|E_{Q_{i,1}}g\|_{L_{\sharp}^{\frac{p(n-1)}{n}}(w_B)}^2)^{1/2}\right]^{p/n},$$
as desired.
\medskip

To prove \eqref{FJbound}, we may assume $Q=[-\delta/2,\delta/2]^{n-1}$, and thus $\widehat{E_Qg}$ will be supported in $[-\delta,\delta]^{n-1}\times[-\delta^2,\delta^2]$. Fix $x=(x_1,\ldots,x_n)$ with $ T_Q(x)\in\F_Q$ and let $y\in 2 T_Q(x)$. Note that $T_Q(x)$ has sides parallel to the coordinate axes. In particular, $y=x+y'$ with $|y_j'|<4\delta^{-1}$ for $1\le j\le n-1$ and $|y_n'|<4\delta^{-2}$. Then
\begin{equation}
\label{h dfhivuhy uytugr9495m8t9349r-=349t-5=9t-0=}
\| E_Q g \|_{L^q(w_{B(y, \delta^{-1})})}^q\lesssim
\end{equation}
$$\int|E_Qg(x_1+u_1,\ldots,x_{n-1}+u_{n-1},x_n+u_n+y_n')|^qw_{B(0, \delta^{-1})}(u)du.$$
Now, using Taylor series we can write
$$|E_Qg(x_1+u_1,\ldots,x_{n-1}+u_{n-1},x_n+u_n+y_n')|$$
$$=|\int\widehat{E_Qg}(\lambda)e(\lambda\cdot(x+u)) e(\lambda_{n}y_{n}')d\lambda|\le$$
$$\le  \sum_{s_{n}=0}^{\infty}\frac{1000^{s_{n}}}{s_{n}!}|\int\widehat{E_Qg}(\lambda)e(\lambda\cdot(x+u))(\frac{\lambda_{n}}{2\delta^{2}})^{s_{n}}d\lambda|$$
$$=\sum_{s_{n}=0}^{\infty}\frac{1000^{s_{n}}}{s_{n}!} |M_{s_n}(E_Qg)(x+u)|.$$
Here $M_{s_n}$ is the operator with Fourier multiplier
$1_{\R^{n-1}}(\lambda_1,\ldots,\lambda_{n-1})m_{s_n}(\frac{\lambda_n}{2\delta^2})$, where
$$m_{s_n}(\lambda_n)=({\lambda_{n}})^{s_{n}}1_{[-1/2,1/2]}(\lambda_{n}).$$
We are able to insert the cutoff because of our initial restriction on the Fourier support of $E_Qg$.

Plugging this estimate into \eqref{h dfhivuhy uytugr9495m8t9349r-=349t-5=9t-0=} we obtain
$$\| E_Q g \|_{L^q_{\sharp}(w_{B(y, \delta^{-1})})}^q\lesssim\sum_{s_{n}=0}^{\infty}\frac{1000^{s_{n}}}{s_{n}!}\|M_{s_n}(E_Qg)\|_{L_\sharp^q(w_{B(x,\delta^{-1})})}^q.$$
Recalling the definition of $F_Q$ and the fact that
$$\delta^n\int_{4B}w_{B(x,\delta^{-1})}(z)dx\lesssim w_B(z),\;z\in\R^n$$
we conclude that
\begin{equation}
\label{hgefcgeryiuftrebyubfureywuyui}
\| F_Q \|_{L^q_{\sharp}(4B)}^q\lesssim  \sum_{s_{n}=0}^{\infty}\frac{1000^{s_{n}}}{s_{n}!}\|M_{s_n}(E_Qg)\|_{L_\sharp^q(w_B)}^q.
\end{equation}

Note that $t\mapsto t^{s_{n}}1_{[-1/2,1/2]}(t)$ agrees on $[-1/2,1/2]$ with a compactly supported smooth function $m_{s_n}^{*}$ defined on $\R$,  with derivatives of any given order uniformly bounded over $s_n$.  It follows that
$$|\widehat{m_{s_n}^*}(x_n)|\lesssim\xi(x_n),$$
with implicit constant independent of $s_n$, where
$$\xi(x_n)\lesssim_M (1+|x_n|)^{-M},$$
for all $M>0$. Let $M_{s_n}^{*}$ denote the operator with multiplier $1_{\R^{n-1}}(\lambda_1,\ldots,\lambda_{n-1})m_{s_n}^{*}(\frac{\lambda_n}{2\delta^2})$.
 We can now write
$$|M_{s_n}(E_Qg)(x)|=|M^{*}_{s_n}(E_Qg)(x)|\lesssim |E_Qg|\odot\xi_{\delta^2}(x)$$
where $\odot$ denotes the convolution  with respect to the last variable $x_n$, and
$$\xi_{\delta^2}(x_n)=\delta^{2}\xi(\delta^2x_n).$$
Using this, one can easily check that
$$\|M_{s_n}(E_Qg)\|_{L^q(w_B)}^q\lesssim \langle |E_Qg|^q\odot\xi_{\delta^2},w_B\rangle$$$$ =\langle |E_Qg|^q,\xi_{\delta^2}\odot w_B\rangle\lesssim \langle |E_Qg|^q,w_B\rangle.$$
Combining this with \eqref{hgefcgeryiuftrebyubfureywuyui} leads to the proof of \eqref{FJbound}
$$\| F_Q \|_{L^q_{\sharp}(4B)}^q\lesssim \sum_{s_{n}=0}^{\infty}\frac{1000^{s_{n}}}{s_{n}!}\|E_Qg\|_{L_\sharp^q(w_B)}^q$$
$$\lesssim \|E_Qg\|_{L_\sharp^q(w_B)}^q.$$
The argument is now complete.

\end{proof}

\section{The iteration scheme}
\label{sec:it}
Let $0<\nu<1$. Throughout this section we fix some $0<\delta<1$ and also $n$  $\nu$-transverse cubes $Q_1,\ldots,Q_n\subset[0,1]^{n-1}$ with side length at least $\delta$.

For a positive integer $s$, $B^s$ will refer to cubes in $\R^n$ with side length $l(B^s)=\delta^{-s}$ and arbitrary centers. We will only encounter cubes $B\subset\R^n$ with side length $l(B)\in 2^{\N}$. This will allow us to perform decompositions using cubes of smaller size in $2^{\N}$.

The implicit constants will be independent of $\delta$,  $g$ and the spatial cubes $Q_i$.
\medskip

Let $t,p\ge 1$ and consider the positive integers $q\le s \le r$.
We define
$$D_t(q,B^r,g)=\big[\prod_{i=1}^{n}(\sum_{Q_{i,q}\in \Di_{\delta^q}(Q_i)}\|E_{Q_{i,q}}g\|_{L^{t}_{\sharp}(w_{B^r})}^2)^{1/2}\big]^{\frac{1}{n}
}.$$

To simplify notation, we will denote by $\B_s(B^r)=\Di_{\delta^{-s}}(B^r)$ the (unique) cover of $B^r$ with  cubes $B^s$ of side length $\delta^{-s}$. Define
$$A_p(q,B^r,s,g)=\big(\frac1{|\B_s(B^r)|}\sum_{B^s\in\B_s(B^r)}D_2(q,B^s,g)^p\big)^{\frac1p}.$$
The letter $A$ will remind us that we have an average.  Note that when $r=s$,
$$A_p(q,B^r,r,g)=D_2(q,B^r,g).$$
For $\frac{2n}{n-1}\le p$, let  $0\le \kappa_p\le 1$ satisfy
$$\frac{n}{p(n-1)}=\frac{1-\kappa_p}{2}+\frac{\kappa_p}{p}.$$
In other words,
$$\kappa_p=\frac{pn-p-2n}{(p-2)(n-1)}.$$
\bigskip
Set also $\kappa_p=0$ for $2\le p\le \frac{2n}{n-1}$.

The next proposition will combine our main two decoupling devices, Theorem \ref{nt4} and the $L^2$ decoupling. The result is a partial decoupling. Indeed, note that the term $A_p(1,B^2,1,g)$ in  \eqref{jhfvheruyg7erigt7t86} involves frequency  cubes of size $\delta$, while the term $A_p(2,B^2,2,g)$ involves frequency  cubes of smaller size $\delta^2$. Inequality \eqref{jhfvheruyg7erigt7t86} is only a partial decoupling in the range $p>\frac{2n}{n-1}$, since  the weight $\kappa_p$ of the term $D_p(1,B^2,g)$ is nonzero. But this weight is zero  when $p\le \frac{2n}{n-1}$. For these values of $p$, inequality \eqref{jhfvheruyg7erigt7t86} has the very simple form
$$A_{p}(1,B^2,1,g)\lesssim_{\epsilon,\nu}\delta^{-\epsilon}A_{p}(2,B^2,2,g).$$
This can be easily iterated and leads to a simpler proof of Theorem \ref{main:steamline} in the range $2\le p\le \frac{2n}{n-1}$. See the discussion at the end of this section.
\begin{proposition}
\label{hegfuyefuyrufyuryfuryowqiuoiuwdoiwd9438u9uirhcfj}
We have for each  $B^2$ and $p\ge 2$
\begin{equation}
\label{jhfvheruyg7erigt7t86}
A_p(1,B^2,1,g)\lesssim_{\epsilon,\nu}\delta^{-\epsilon}A_p(2,B^2,2,g)^{1-\kappa_p}D_p(1,B^2,g)^{\kappa_p}.
\end{equation}
\end{proposition}

\begin{proof}
Assume first that $p\ge \frac{2n}{n-1}$. By H\"older,  $$\|E_{Q_{i,1}}g\|_{L^{2}_{\sharp}(w_{B^1})}\lesssim \|E_{Q_{i,1}}g\|_{L^{\frac{p(n-1)}{n}}_{\sharp}(w_{B^1})}.$$
Using this and  Theorem \ref{nt4},  we can write
\begin{equation}
A_p(1,B^2,1,g)\lesssim_{\epsilon,\nu}\delta^{-\epsilon} (\prod_{i=1}^{n}\sum_{Q_{i,1}\in\Di_{\delta}(Q_i)}\|E_{Q_{i,1}}g\|_{L_{\sharp}^{\frac{p(n-1)}n}(w_{B^2})}^2)^{\frac1{2n}}.
\end{equation}

Using  H\"older's inequality we can dominate this by
\begin{equation}
\label{dgchgdshgfdsfkjdsfhkjfdhgjkfdhgjfdkh}
\le (\prod_{i=1}^{n}\sum_{Q_{i,1}\in\Di_{\delta}(Q_i)}\|E_{Q_{i,1}}g\|_{L_{\sharp}^{2}(w_{B^2})}^2)^{\frac{1-\kappa_p}{2n}}(\prod_{i=1}^{n}\sum_{Q_{i,1}\in\Di_{\delta}(Q_i)}
\|E_{Q_{i,1}}g\|_{L_{\sharp}^{p}(w_{B^2})}^2)^{\frac{\kappa_p}{2n}}.
\end{equation}
It suffices now to apply $L^2$ decoupling (Proposition \ref{nl3}) to the first term in \eqref{dgchgdshgfdsfkjdsfhkjfdhgjkfdhgjfdkh}.
\medskip

We have ``interpolated" between $L^2$ and $L^p$. We have used $L^2$ because - as explained in Section \ref{L2} - this space facilitates  the most efficient decoupling. Indeed, note that the term  $A_p(2,B^2,2,g)$ on the right hand side of \eqref{jhfvheruyg7erigt7t86} has cubes of side length $\delta^2$, which is as small as one can hope for, given the size of the spatial cube $B^2$.
\medskip

If $p<\frac{2n}{n-1}$, using \eqref{jhfvheruyg7erigt7t86} with $p=\frac{2n}{n-1}$ we can write
$$A_p(1,B^2,1,g)\le A_{\frac{2n}{n-1}}(1,B^2,1,g)\lesssim_{\epsilon,\nu}\delta^{-\epsilon}  A_{\frac{2n}{n-1}}(2,B^2,2,g) =\delta^{-\epsilon}A_p(2,B^2,2,g).$$
\end{proof}
 Inequality \eqref{jhfvheruyg7erigt7t86} is easily seen to be true with $\kappa_p$ replaced with 1, by simply invoking \eqref{jgiou tiguyt reym89r--9048-0023=-1r90-8=4-293=-0} and the fact that $D_2\lesssim D_p $. Consequently, it will be true for each exponent in the interval $[\kappa_p,1]$. The example $g=1_Q$ with $l(Q)=\delta$ shows that one can not consider exponents smaller than $\kappa_p$.
The relevant thing about $\kappa_p$ that will be used in the final section is the fact that $\kappa_p<\frac12$ precisely in the subcritical range $p<\frac{2(n+1)}{n-1}$.

\bigskip

The following sequence of propositions will allow us to rewrite  \eqref{jhfvheruyg7erigt7t86} in a form that is more suitable for iteration.
\begin{proposition}
\label{hegfuyefuyrufyuryfuryowqiuoiuwdoiwd9438u9uirhcfjsec}
We have for each  cube $B^{M}$ with $M\ge 2$ and $p\ge 2$
\begin{equation}
\label{jhfvheruyg7erigt7t86sec}
A_p(1,B^M,1,g)\lesssim_{\epsilon,\nu}\delta^{-\epsilon}A_p(2,B^M,2,g)^{1-\kappa_p}D_p(1,B^M,g)^{\kappa_p}.
\end{equation}
The implicit constant is independent of $M$.
\end{proposition}
\begin{proof}Raise \eqref{jhfvheruyg7erigt7t86} to the power $p$, sum over all cubes $B^2\in\B_2(B^M)$ and use H\"older's inequality
$$\|(a_jb_j)_j\|_{l^1}\le \|(a_j)_j\|_{l^{\frac{1}{1-\kappa_p}}}\|(b_j)_j\|_{l^{\frac{1}{\kappa_p}}}.$$
The only thing that needs to be verified is the inequality
\begin{equation}
\label{rejiogutiugtupgroofpqowefp[owef]-op}
\sum_{B^2\in\B_2(B^M)}D_p(1,B^2,g)^p\lesssim D_p(1,B^M,g)^p.
\end{equation}
This however immediately follows from Minkowski's inequality (recall $p\ge 2$) and the fact that
$$\sum_{B^2\in\B_2(B^M)}w_{B^2}\lesssim w_{B^M}.$$
\end{proof}

\begin{proposition}
\label{hegfuyefuyrufyuryfuryowqiuoiuwdoiwd9438u9uirhcfjth}
Let $l,m\in \N$ with $l+1\le m$.
We have for each  cube $B^{2^m}$  and $p\ge 2$
\begin{equation}
\label{jhfvheruyg7erigt7t86th}
A_p(2^l,B^{2^m},2^l,g)\lesssim_{\epsilon,\nu}\delta^{-2^{l}\epsilon}A_p(2^{l+1},B^{2^m},2^{l+1},g)^{1-\kappa_p}D_p(2^l,B^{2^m},g)^{\kappa_p}.
\end{equation}
The implicit constant is independent of $l,m$.
\end{proposition}
\begin{proof}Apply \eqref{jhfvheruyg7erigt7t86sec} with $\delta$ replaced by $\delta^{2^l}$ and $M=2^{m-l}$.

\end{proof}
We can now iterate Proposition \ref{hegfuyefuyrufyuryfuryowqiuoiuwdoiwd9438u9uirhcfjth} to get the following immediate conclusion.
\begin{proposition}
\label{hegfuyefuyrufyuryfuryowqiuoiuwdoiwd9438u9uirhcfjfo}If $m\ge 1$ and $p\ge 2$
\begin{equation}
\label{jhfvheruyg7erigt7t86fo}
A_p(1,B^{2^m},1,g)\lesssim_{\epsilon,\nu,m}\delta^{-\epsilon}A_p(2^{m-1},B^{2^m},2^{m-1},g)^{(1-\kappa_p)^{m-1}}\prod_{l=0}^{m-2}D_p(2^l,B^{2^m},g)^{\kappa_p(1-\kappa_p)^{l}}.
\end{equation}

\end{proposition}
The implicit constant is now allowed to depend on  $m$, but this dependence will prove to be completely harmless.
\bigskip

We close this section with a quick proof of
$$\D_n(\delta,p)\lesssim_{\epsilon}\delta^{-\epsilon}$$
for $2\le p\le \frac{2n}{n-1}$. This fact was first proved in \cite{Bo2}.
In this range $\kappa_p=0$ and \eqref{jhfvheruyg7erigt7t86fo} becomes a very satisfactory inequality
$$A_{p}(1,B^{2^m},1,g)\lesssim_{\epsilon,\nu,m}\delta^{-\epsilon}A_{p}(2^{m-1},B^{2^m},2^{m-1},g).$$
Combining this with \eqref{m8tvuy45t89568y0695y076o0-340}
we may write
$$\||\prod_{i=1}^n[\sum_{Q_{i,1}\in\Di_{\delta}(Q_i)}|E_{Q_{i,1}g}|^2]^{1/2n}\|_{L^{p}(B^{2^m})}\lesssim_{\epsilon,\nu,m}\delta^{-\epsilon} D_{p}(2^{m-1},B^{2^m}).$$
By invoking Cauchy--Schwarz, we can afford a rather trivial decoupling
$$\||\prod_{i=1}^nE_{Q_ig}|^{1/n}\|_{L^{p}(B^{2^m})}\le \delta^{-(n-1)/2}\||\prod_{i=1}^n[\sum_{Q_{i,1}\in\Di_{\delta}(Q_i)}|E_{Q_{i,1}g}|^2]^{1/2n}\|_{L^{p}(B^{2^m})}.$$
Combining these two and substituting $\delta^{2^m}\mapsto \delta$ we can write
$$\||\prod_{i=1}^nE_{Q_ig}|^{1/n}\|_{L^{p}(B^{1})}\lesssim_{\epsilon,\nu,m}\delta^{-\epsilon}\delta^{-(n-1)2^{-m-1}}D_{p}(\frac12,B^{1}).$$

Choose now $m$ as large as desired to argue that
$$\D_n(\delta,p,\nu)\lesssim_{\epsilon,\nu}\delta^{-\epsilon}.$$
Finally, combine this with Theorem \ref{ch ft7wbtfgb6n17r782brym9,iqmivpk[l} using induction on $n$ to argue that
$$\D_n(\delta,p)\lesssim_{\epsilon}\delta^{-\epsilon}.$$
\bigskip

Now back to the case $p>\frac{2n}{n-1}$. As mentioned earlier, \eqref{jhfvheruyg7erigt7t86fo} is only a partial decoupling in this range.
 The argument for this case presented in the next section will go as follows. Assume the linear decoupling constant satisfies $\D_n(\delta,p)\sim \delta^{-\eta_p}$. We will first apply parabolic rescaling to majorize the terms $D_p$ in \eqref{jhfvheruyg7erigt7t86fo} by some powers of $\delta^{-\eta_p}$.
Then we will combine \eqref{jhfvheruyg7erigt7t86fo} with a trivial decoupling (Cauchy--Schwarz)  to derive an upper bound on the multilinear constant $\D_n(\delta,p,\nu)$ in terms of $\delta^{-\eta_p}$. We play this against  Theorem \ref{ch ft7wbtfgb6n17r782brym9,iqmivpk[l}, which produced a lower bound for $\D_n(\delta,p,\nu)$ involving $\delta^{-\eta_p}$. These will force $\eta_p$ to be zero.

\bigskip

\section{The final argument}
\bigskip

In this section we present the details for the proof of Theorem \ref{main:steamline}. Let $E\ge 100n$. By combining the triangle and Cauchy-Schwarz inequalities we find that $\D_n(\delta,p,E)\lesssim \delta^{-C_p}$, for some $C_p$ large enough.
 For $p\ge 2$ let $\eta_{p,n,E}=\eta_{p,E}\ge 0$ be the unique (finite) number such that
\begin{equation}
\label{9855e1}
\lim_{\delta\to 0}\D_n(\delta,p,E)\delta^{\eta_{p,E}+\sigma}=0,\text{ for each }\sigma>0
\end{equation}
and
\begin{equation}
\label{9855e2}
\limsup_{\delta\to 0}\D_n(\delta,p,E)\delta^{\eta_{p,E}-\sigma}=\infty,\text{ for each }\sigma>0.
\end{equation}

We will use induction on $n$, as described at the end of Section \ref{brief}. Assume either that $n=2$, or that $n\ge 3$ and that in addition we have
$$
\D_{n-1}(\delta,p,E)\lesssim_\epsilon\delta^{-\epsilon}
$$
for $E\ge 100(n-1)$ and $2\le p\le \frac{2n}{n-2}$.
We need to prove that
$\eta_{p,n,E}=0$ for $E\ge 100n$ and $2\le p\le \frac{2(n+1)}{n-1}$.
Note that for such $p$ we automatically have that $p$ is smaller than $\frac{2n}{n-2}$, the critical index for decouplings in $\R^{n-1}$. In particular, if $n\ge 3$ our induction hypothesis guarantees that
\begin{equation}
\label{fjdnhytg7879ro3[0r9i845689409er87y76t76r6}
\D_{n-1}(\delta,p,E)\lesssim_\epsilon\delta^{-\epsilon}
\end{equation}
for each $E\ge 100n$ and $2\le p\le \frac{2(n+1)}{n-1}$.

Fix $2\le p<\frac{2(n+1)}{n-1}$ for the rest of the proof (so in particular we have $p\le 20$). Fix also $E\ge 100n$. All quantities $A_p$ and $D_p$ will be implicitly assumed to be relative to this $E$. The case $p=\frac{2(n+1)}{n-1}$ will follow via a standard limiting argument explained in the end of the section. Note that for $2\le p<\frac{2(n+1)}{n-1}$  we have
\begin{equation}
\label{eriuf8uy9-oo4tu568uy0i6950y890789}
2(1-\kappa_p)>1.
\end{equation}

We start with the following rather immediate consequence of Proposition \ref{hegfuyefuyrufyuryfuryowqiuoiuwdoiwd9438u9uirhcfjfo}.

\begin{theorem}\label{dsjfgyetfy3ycqedbf6reudiqtdyteyeuqwhgahg}
Consider  $n$  $\nu$-transverse cubes $Q_1,\ldots,Q_n\subset[0,1]^{n-1}$ with side length at least $\delta$. Then for $m\ge 1$
 and $p\ge 2$ we have $$A_p(1,B^{2^m},1,g)\lesssim_{\epsilon,\nu,m}$$$$\delta^{-(\eta_p+\epsilon)(2^m-\frac{2\kappa_p}{2\kappa_p-1}+\frac{(2(1-\kappa_p))^m}{2\kappa_p-1})}D_p(2^{m-1},B^{2^m},g),$$
with the implicit constant independent of $Q_i$.
\end{theorem}
\begin{proof}This will follow from Proposition \ref{hegfuyefuyrufyuryfuryowqiuoiuwdoiwd9438u9uirhcfjfo}, once we make a few observations.

First,
\begin{equation}
\label{m8tvuy45t89568y0695y076o0-340}
A_p(2^{m-1},B^{2^m},2^{m-1},g)\lesssim D_p(2^{m-1},B^{2^m},g).
\end{equation}
This is a consequence of H\"older, Minkowski's inequality in $l^{\frac{p}2}$ and \eqref{jgiou tiguyt reym89r--9048-0023=-1r90-8=4-293=-0} (very much like \eqref{rejiogutiugtupgroofpqowefp[owef]-op}).

Second, an application of Proposition \ref{propo:parabooorescal} shows that
$$D_p(2^l,B^{2^m},g)\lesssim \D_n(\delta^{2^m-2^{l+1}},p)D_p(2^{m-1},B^{2^m},g).$$
Finally, combine these with \eqref{9855e1} and Proposition \ref{hegfuyefuyrufyuryfuryowqiuoiuwdoiwd9438u9uirhcfjfo}.

\end{proof}

By replacing $\delta^{2^m}$ with $\delta$, we prefer to write the inequality in Theorem  \ref{dsjfgyetfy3ycqedbf6reudiqtdyteyeuqwhgahg} as follows
$$A_p(2^{-m},B^{1},2^{-m},g)\lesssim_{\epsilon,\nu,m}$$
\begin{equation}
\label{dhvfdhjgvjdkshcajdshgfjhskdhdjfklwkadlhdfjghjkfdjb}
\delta^{-(\eta_p+\epsilon)(1-2^{-m}\frac{2\kappa_p}{2\kappa_p-1}+\frac{(1-\kappa_p)^m}{2\kappa_p-1})}D_p(\frac12,B^{1},g),
\end{equation}
with the implicit constant independent of the cubes $Q_i$. Here the assumption is $l(Q_i)\ge \delta^{2^{-m}}$.

\bigskip

 Let $B=B^1$ be a  cube in $\R^n$ with $l(B)=\delta^{-1}$. Consider  $n$  $\nu$-transverse cubes $Q_1,\ldots,Q_n\subset[0,1]^{n-1}$ with side length $\mu\ge \delta^{2^{-m}}$. Let as before $\Di_{\mu^{-1}}(B)$ denote the partition of $B$ using cubes $\Delta$ with $l(\Delta)=\mu^{-1}$. Denote also by $\B_m(B)$  the partition of $B$ using cubes $\Delta_m$ with $l(\Delta_m)=\delta^{-2^{-m}}$.
\medskip

We may write, first by combining Cauchy--Schwarz and \eqref{jgiou tiguyt reym89r--9048-0023=-1r90-8=4-293=-0}

\begin{equation}
\label{ fu9ety8290-78-=1=23 -r02v9=--e322}
[\frac1{|\B_{\mu^{-1}}(B)|}\sum_{\Delta\in\Di_{\mu^{-1}}(B)}(\prod_{i=1}^n\|E_{Q_i}g\|_{L^p_\sharp(w_{\Delta,10E})}^p)^{\frac{1}n}]^{\frac1p}\lesssim 
\end{equation}
\begin{equation}
\label{ fu9ety8290-78-=1=23 -r02v9=--e32}
\delta^{-(n-1)2^{-m-1}}[\frac1{|\B_m(B)|}\sum_{\Delta_m\in\B_m(B)}(\prod_{i=1}^n\|(\sum_{q_i\in\Di_{\delta^{2^{-m}}}(Q_i)}|E_{q_i}g|^{2})^{1/2}
\|_{L^p_\sharp(w_{\Delta_m,10E})}^p)^{\frac{1}n}]^{\frac1p},
\end{equation}
then using Minkowski's inequality and \eqref{ nghbugtrt90g0-er9t-9} (recall that $p\le 20$)
\begin{equation}
\label{ fu9ety8290-78-=1=23 -r02v9=--e321}
\le \delta^{-(n-1)2^{-m-1}}[\frac1{|\B_m(B)|}\sum_{\Delta_m\in\B_m(B)}(\prod_{i=1}^n\sum_{q_i\in\Di_{\delta^{2^{-m}}}(Q_i)}\|E_{q_i}g\|_{L^p_\sharp(w_{\Delta_m,10E})}^{2})^{\frac{p}{2n}}]^{\frac1p}
\end{equation}
$$\le \delta^{-(n-1)2^{-m-1}}[\frac1{|\B_m(B)|}\sum_{\Delta_m\in\B_m(B)}(\prod_{i=1}^n\sum_{q_i\in\Di_{\delta^{2^{-m}}}(Q_i)}\|E_{q_i}g\|_{L^2_\sharp(w_{\Delta_m,E})}^{2})^{\frac{p}{2n}}]^{\frac1p}$$
$$=A_p(2^{-m},B^{1},2^{-m},g).$$
Invoking \eqref{dhvfdhjgvjdkshcajdshgfjhskdhdjfklwkadlhdfjghjkfdjb} and removing the normalization, we conclude that
$$[\sum_{\Delta\in\Di_{\mu^{-1}}(B)}(\prod_{i=1}^n\|E_{Q_i}g\|_{L^p(w_\Delta,10E)}^p)^{\frac{1}n}]^{\frac1p}
\lesssim_{\epsilon,\nu,m}$$$$
\delta^{-(\eta_p+\epsilon)(1-2^{-m}\frac{2\kappa_p}{2\kappa_p-1}+\frac{(1-\kappa_p)^m}{2\kappa_p-1})}\delta^{-(n-1)2^{-m-1}}\left[\prod_{i=1}^n\sum_{q_i\in\Di_{\delta^{1/2}}(Q_i)}\|E_{q_i}g\|_{L^p(w_{B,E})}^2\right]^{\frac1{2n}}.$$

By taking a supremum over all $Q_i$, $B$, $g$ as above, we deduce the following inequality, which is a stronger substitute for \eqref{hjgvhdgfyrtefuiyecuryueuyguirtygio}
\begin{equation}
\label{dschjgdshdcyfyuvtioduewuyt}
\D_n(\delta,p,\nu,m,E)\lesssim_{\epsilon,\nu,m}
\delta^{-(\eta_p+\epsilon)(1-2^{-m}\frac{2\kappa_p}{2\kappa_p-1}+\frac{(1-\kappa_p)^m}{2\kappa_p-1})}\delta^{-(n-1)2^{-m-1}}.
\end{equation}
Combining this with Theorem \ref{ch ft7wbtfgb6n17r782brym9,iqmivpk[l} (use \eqref{fjdnhytg7879ro3[0r9i845689409er87y76t76r6}) and \eqref{9855e2} we may now write
$$\delta_l^{-\eta_p+\epsilon+\epsilon(\nu)}\lesssim_{\epsilon,\nu,m}
\delta_l^{-(\eta_p+\epsilon)(1-2^{-m}\frac{2\kappa_p}{2\kappa_p-1}+\frac{(1-\kappa_p)^m}{2\kappa_p-1})}\delta_l^{-(n-1)2^{-m-1}}$$
for some sequence $\delta_l$ converging to zero. This in turn forces
$$-\eta_p+\epsilon+\epsilon(\nu)\ge -(\eta_p+\epsilon)(1-2^{-m}\frac{2\kappa_p}{2\kappa_p-1}+\frac{(1-\kappa_p)^m}{2\kappa_p-1})-(n-1)2^{-m-1}$$
for each $\epsilon,\nu>0$. Thus, letting $\epsilon,\nu\to 0$ we get
$$-\eta_p\ge -\eta_p(1-2^{-m}\frac{2\kappa_p}{2\kappa_p-1}+\frac{(1-\kappa_p)^m}{2\kappa_p-1})-(n-1)2^{-m-1},$$
and by rearranging terms
\begin{equation}
\label{jdhfkfdhkfkjhgfsdlifhburehro3ytyewyuiweyu}
(n-1)2^{-1}\ge \eta_p\frac{[2(1-\kappa_p)]^m-2\kappa_p}{1-2\kappa_p}.
\end{equation}
As this holds for each $m\ge 1$, \eqref{eriuf8uy9-oo4tu568uy0i6950y890789} will immediately force $\eta_p=0$.
\medskip

Let us now show that $\eta_{p_n}=0$ for $p_n=\frac{2(n+1)}{n-1}$. Let $B\subset \R^n$ be a cube with  $l(B)=\delta^{-1}$.
Using inequality \eqref{ nghbugtrt90g0-er9t-9}, for  $p<p_n$ we can write
$$\|E_{[0,1]^{n-1}}g\|_{L^{p_n}({B})}\lesssim \|E_{[0,1]^{n-1}}g\|_{L^{p}(w_{B})}.$$
Combining this with H\"older's inequality we get
$$\|E_{[0,1]^{n-1}}g\|_{L^{p_n}({B})}\lesssim \D_{n}(\delta,p)(\sum_{Q\in\Di_{\delta^{1/2}}([0,1]^{n-1})}\|E_Qg\|_{L^{p}(w_B)}^2)^{1/2}$$
$$\lesssim_\epsilon \delta^{-\epsilon}\|\textbf{1}\|_{L^{\frac{q}{q-1}}(w_B)}(\sum_{Q\in\Di_{\delta^{1/2}}([0,1]^{n-1})}\|E_Qg\|_{L^{p_n}(w_B)}^2)^{1/2}.$$
Note that $q\to 1$ as $p\to p_n$. Finally, invoke Lemma \ref{nl9}, cf. Remark \ref{aha}.

\end{document}